\documentclass[a4paper,12pt]{amsart}

\usepackage{amsmath,amsfonts,latexsym,amssymb,amscd}
\usepackage{amsthm,latexsym,epic,eepic,amsbsy,amstext,amsxtra,latexsym, xypic}
\usepackage{CJK,pshan}
\usepackage{indentfirst}
\usepackage{epsfig}
\usepackage{texdraw}
\usepackage{color}
\usepackage{tikz}
\usepackage{tikz-cd}
\usetikzlibrary{er}
\usetikzlibrary{positioning}
\usetikzlibrary{arrows}
\usepackage{mathtools}

\usetikzlibrary{shapes}

\usetikzlibrary{matrix}
\theoremstyle{plain}
\newtheorem{thm}{Theorem}[section]
\newtheorem{theorem}[thm]{Theorem}

\newtheorem{lemma}[thm]{Lemma}
\newtheorem{corollary}[thm]{Corollary}
\newtheorem{proposition}[thm]{Proposition}
\theoremstyle{definition}
\newtheorem{remark}[thm]{Remark}

\newtheorem{notation}[thm]{Notation}
\newtheorem{notation-definition}[thm]{Notation-Definition}
\newtheorem{notation-remark}[thm]{Notation-Remark}

\numberwithin{equation}{section}

\newcommand{\CC}{\mathbb{C}}

\newcommand{\PP}{\mathbb{P}}

\newcommand{\cO}{\mathcal{O}}
\newcommand{\Sym}{{\rm Sym}}

\newcommand{\be}{{\bold e}}

\begin{document}

\title[Lagrangian fibration structure on the cotangent bundle]{Lagrangian fibration structure on the cotangent bundle of a del Pezzo surface of degree 4}
\author{Hosung Kim and Yongnam Lee}

\address {Center for Complex Geometry\\
Institute for Basic Science (IBS)\\
55 Expo-ro, Yuseong-gu\\ 
Daejeon, 34126 Korea}
\email{hosung@ibs.re.kr}

\address {Center for Complex Geometry\\
Institute for Basic Science (IBS)\\
55 Expo-ro, Yuseong-gu\\ 
Daejeon, 34126 Korea,  and 
\newline\hspace*{3mm} Department of Mathematical Sciences\\
KAIST\\
291 Daehak-ro, Yuseong-gu\\ 
Daejeon, 34141 Korea}
\email{ynlee@ibs.re.kr}

\date{\today}

\begin{abstract}
In this paper, we show that there is a natural Lagrangian fibration structure on the map $\Phi$ from the cotangent bundle of a del Pezzo surface $X$ of degree 4 to $\CC^2$. Moreover, we describe explicitly all level surfaces of the above natural map $\Phi$.
\end{abstract}

\maketitle

\section{Introduction}

Throughout this paper we will work over the field of complex numbers.

The cotangent bundle of a complex projective manifold carries a natural holomorphic symplectic 2-form. The existence of a natural Lagrangian fibration structure of these non-compact complex manifolds has not been studied very much. 

There are two famous known examples in this direction, one is the Hitchin map \cite{Hit} 
$h : T^*_X \to \CC^{(r^2-1)(g-1)}$ where $X$ is the moduli space $SU^s_C(r,d)$ of stable vector bundle of rank $r$ with a fixed determinant of degree $d$ coprime to $r$ over a smooth projective curve $C$ of genus $g$. This map has been used as a tool to derive results on the moduli spaces themselves in \cite{BNR}. The other example is a rational homogeneous space $G/P$ where $G$ is a semisimple complex Lie group and $P$ is a parabolic subgroup. The group $G$ acts on the cotangent bundle $T^*_{G/P}$ as symplectic automorphisms. This induces the moment map $ T^*_{G/P}\to \mathcal{G}^*$ to the dual of the Lie algebra of $G$ (cf. Section 1.4 of [CG]). Both examples $SU^s_C(r,d)$ and $G/P$ are Fano manifolds. This suggests that there may exist some interesting Lagrangian fibration structure in the cotangent bundles of Fano manifolds. J-M. Hwang \cite{Hwang} shows that the varieties of minimal rational tangents play an important role in the symplectic geometry of the cotangent bundles of uniruled projective manifolds.

The current paper is motivated by the fundamental work of the moduli spaces of vector bundles from the viewpoint of symplectic geometry of its cotangent bundle by Hitchin \cite{Hit}, and by the result of J-M. Hwang and Ramanan in \cite{HR04}, where they studied  the Hitchin system and the Hitchin discriminant associated to the Hitchin map on the cotangent bundle of $SU^s_C(r,d)$.

Our computational result directly shows that the cotangent bundle of a del Pezzo surface $X$ of degree 4 has also the Lagarangian fibration structure $\Phi:T_X^*\rightarrow \mathbb C^2$, and its level surfaces have some similar properties of the Hitchin discriminant in \cite{HR04}. Similarly as Corollary 4.6 in \cite{HR04}, $\Phi^{-1}(\Delta)$ is the closure of the union of rational curves in $T_X^*$ where $\Delta=\{b\in\CC^2\, | \,\text{$\Phi^{-1}(b)$ is singular}\}$. In the current paper, $\Delta=\{\text{five lines through the origin in $\CC^2$}\}$.

On the other hand, the positivity problem of the tangent bundle of a del Pezzo surface $S$ of degree $d$ is completely answered recently in \cite{Mal21} and in \cite{HLS}. If $S$ is a del Pezzo surafce of degree $d$ then
\begin{itemize}
    \item $T_S$ is big if and only if $d\ge 5$.
    \item $H^0(S, {\rm Sym}^m T_S)=0$ for all $m\ge 1$ if and only if $d\le 3$.
\end{itemize}
So the case of $d=4$ arouses special interest to us.

Let $X$ be a del Pezzo surface of degree 4. Then $X$ is a complete intersection of two  hypersurfaces in $\mathbb P^4=\mathbb P^4_{y_1,\ldots,y_5}$ defined by homogeneous polynomials $Q_1$ and $Q_2$ of degree 2  in  variables $y_1,\ldots,y_5$ respectively. 
By  a linear change of variables and multiplication by $\mathbb C^*$, we can assume that 
\begin{equation}\label{eq.quadrics}Q_1=\sum_{i=1}^5 y_i^2
\mbox{ and } Q_2=\sum_{i=1}^5a_iy_i^2 \end{equation}for some distinct  $a_i\in\mathbb C$. (cf. Theorem 8.6.2 in \cite{Dol}). 

From Theorem 5.1 in \cite{OL19} and the proof of Theorem 6.1 in \cite{Mal21}, there is an isomorphism of  graded rings:

\begin{equation}\label{eq.graded ring isomorphism}\bigoplus_{m=0}^{\infty}H^0(X, \Sym^mT_X)\simeq \mathbb C[Q_1,Q_2].\end{equation}
In particular $Q_1$ and $Q_2$ form a  basis of $H^0(X, \Sym^2 T_X)$. 

Let  
$$ \Phi:T^*_X\rightarrow \mathbb C^2$$
be the natural morphism defined by  the pair  $(Q_1,Q_2)$. For each $e\in \mathbb C^2$,  the fiber $\Phi^{-1}({e})\subset T_X^*$ will be called a level surface, and we will denote it by $S_e$. From the isomorphism in  (\ref{eq.graded ring isomorphism}), we can see that $S_{e}\cong S_{\lambda e}$ for all $\lambda\in\mathbb C^*$. 

First of all, in this paper, we show the following theorem. The proof will be given in Section 2.

\begin{theorem}\label{t.lagrangian}
The morphism $\Phi:T_X^*\rightarrow \mathbb C^2$ is a Lagrangian fibration.
\end{theorem}

It means that the restriction $\omega|_{S_e}$ of the natural symplectic two form $\omega$ on $T_X^*$ is zero.
Remark~\ref{relation} also explains the relation between the Lagrangian fibration structure of the Hitchin map for the case of $g=2$ and the Lagrangian fibration structure of the cotangent bundle of a del Pezzo surface $X$ of degree 4.
\medskip

Meanwhile, we let $\zeta:=\mathcal O(1)$ be the tautological line bundle on $\mathbb P(T_X)$ so that $\pi_*\zeta=T_X$ where $\pi: \PP(T_X) \to X$ be the projection. 
By the isomorphism  $H^0(\mathbb P(T_X),2\zeta)\simeq H^0(X,\Sym^2 T_X)$, 
the pencil $\{Q_{{\be}}\}_{ \be\in \PP^1}$ of quadric hypersurfaces in $\mathbb P^4$ induced by  $Q_1$ and $Q_2$ gives the linear system $|2\zeta|$ in $\PP(T_X)$ defining a rational map $\tilde \phi:\mathbb P(T_X)\dashrightarrow \mathbb P^1$. 

It is well known that there are exactly 16 lines $\ell_1,\ldots,\ell_{16}$ in $\mathbb P^4$ contained in $X$.   
The base locus $B$ of the linear system $|2\zeta|$ in $\PP(T_X)$ consists of the disjoint union of 16 sections $\ell_i'$ of $\mathbb P(T_X|_{\ell_i})\rightarrow \ell_i$ which are associated to  quotients 
$T_X|_{\ell_i}=\mathcal O_{\mathbb P^1}(2)\oplus \mathcal O_{\mathbb P^1}(-1)\twoheadrightarrow \mathcal O_{\mathbb P^1}(-1)$ (p.12 in \cite{HLS}); Since $\ell_i'\cdot \zeta=-1$, we have  $\ell_i'\subset B$.

After the blow-up $\mu_B:{\rm Bl}_B\mathbb P(T_X)\rightarrow \mathbb P(T_X)$ along the base locus $B$, we have a morphism  
$$\phi:{\rm Bl}_B\mathbb P(T_X)\rightarrow \mathbb P^1,$$
and the following commutative diagram of morphisms and rational maps:

\[\xymatrix @R=1pc @C=2pc {
{\rm Bl}_B\mathbb P(T_X)\ar[r]^{\mu_B}\ar[rdd]_{\phi}& \mathbb P(T_X)\ar[rd]_{\pi}\ar@{.>}^{\tilde \phi}[dd]&  &T_X^*\ar@{.>}[ll]\ar[dd]^{\Phi}\ar[ld]^{\Pi}\\
&  &X & \\
& \mathbb P^1 & &\mathbb C^2\ar@{.>}[ll]
}
\]

\bigskip

For each ${\bold e}\in \mathbb P^1$, we let $K_{{\bold{e}}}$ be the fiber $\phi^{-1}({\bold{e}}).$ In Lemma~\ref{l.characterization of irreducible fibers} we show that  $K_{\be}$  is a double cover of $X$ if $Q_{\be}$ is smooth. Due to the description of Section 2 in \cite{OL19}, for each $x\in X$, the points in $K_\be$ over $x$ correspond to lines in $Q_\be\cap {\bold T}_x X$ through $x$ where ${\bold T}_x X\subset \PP^4$ denotes the embedded projective tangent plane to $X$ at $x$. So if $Q_{\be}$ is smooth, the branch locus of the double cover from $K_\be$ to $X$  is the locus of points $x$ such that $Q_\be\cap {\bold T}_x X$ is a double line. The total dual VMRT theory (cf. \cite{HLS}) also helps to give an explicit description of $K_\be$, especially when $K_\be$ is not irreducible. We will explain this description in Section 3.

Through an explicit description of $K_\be$,  we get the following theorem (a) for a general $e \in \CC^2$ and (b). Then by using the idea of the characteristic vector fields in \cite{HO09}, the proof of Theorem~\ref{level surface}(=Theorem~\ref{level}) can be completed.
The proof will be given in Section 3.

\begin{theorem}\label{level surface}Let  $\bold b_1,\ldots,\bold b_5$ be the points in $\mathbb P^1$ such that  $Q_{\bold b_i}$ are singular. For each $i=1,\ldots,5$, take one point $b_i$ in the fiber at $\bold b_i$ of 
 the quotient map $\mathbb C^2\setminus\{0\}\rightarrow \mathbb P^1$ and let $\mathbb C\cdot b_i$ be the line in $\mathbb C^2$ through $b_i$ and the origin . 
We have the following description of level surfaces of  $\Phi:T_X^*\to \CC^2$.
\begin{itemize}
\item[(a)] For every $e\in\CC^2\setminus\cup_{i=1}^5 \mathbb C\cdot b_i$, $S_e$ is $\bar S_e\setminus\{{\text{16 points}}\}$ where $\bar S_e$ is isomorphic to the Jacobian variety of a curve $C_e$ of genus two.
\item[(b)] For each $i=1,\ldots,5$, we have  the following description of $S_{b_i}$.
\begin{itemize}
    \item[(i)] $S_{b_i}$ consists of two irreducible components $A_{i, 1}$ and $A_{i, 2}$. 
    \item[(ii)] Each $A_{i, j}$ for $j=1, 2$ is a ruled surface$\setminus\text{\{8 points\}}$ over an elliptic curve $\bar E_{i,j}$.
    \item[(iii)] $A_{i, 1}\cap A_{i, 2}$ is an elliptic curve $E'_{b_i}$.
    \item[(iv)] In the fibration $A_{i, j}\to\bar E_{i,j}$, $E'_{b_i}$ intersects two distinct points at each fiber.
\end{itemize} 
\end{itemize}
\end{theorem}

We remark that $X$ is isomorphic to the blow up of $\mathbb P^2$ at the five points which are the images of $\bold b_1,\ldots,\bold b_5$ under the Veronese embedding (cf. \cite{Sk}). 

As a corollary, the map $\Phi$ is flat, and all elements of the linear system $|2\zeta|$ in $\PP(T_X)$ can be also fully described.


\begin{theorem} We have the following description of $K_\be$ for all $\be\in\PP^1$.
\begin{itemize}
    \item[(a)] For every ${\bold e}\in \mathbb P^1\setminus\{\bold b_1,\ldots,\bold b_5\}$, 
$K_\be$ is a K3 surface of degree 8 of Kummer type. It has 16 (-2)-curves $\ell_{\bold e,i}$  which are intersection of  $K_{\be}$ with the exceptional divisor $D$ of the blow-up $\mu_B:{\rm Bl}_B\mathbb P(T_X)\rightarrow \mathbb P(T_X)$. 
\item[(b)] For each $i=1,\ldots,5$, we have  the following  description of $K_{\bold b_i}$. 
\begin{itemize}
    \item[(i)] $K_{\bold b_i}$ consists of two irreducible components $\breve{\mathcal C}_{i,1}$ and $\breve{\mathcal C}_{i,2}$.
    \item[(ii)] For each $j=1,2$, we have a conic fibration $\pi_{i,j}:X\rightarrow \mathbb P^1$ with four singular fibers such that $\breve{\mathcal C}_{i,j}$  is isomorphic to the blow-up of $X$ at four distinct points which are the singular points of the four singular fibers of  $\pi_{i,j}$. 
    \item[(iii)] $\breve{\mathcal C}_{i,1}\cap \breve{\mathcal C}_{i,2}$ is a smooth elliptic curve $E_{\bold b_i}$.
    \item[(iv)] In the fibration $\breve\pi_{i,j}:\breve{\mathcal C}_{i,j}\to\PP^1$ given by the composition of the blow-up $\breve{\mathcal C}_{i,j}\rightarrow X$ in $(ii)$ and the conic fibration $\pi_{i,j}$, $E_{\bold b_i}$ intersects two distinct points at each smooth fiber, and one point at the exceptional curve of each singular fiber. 
\end{itemize} 
\end{itemize}
\end{theorem}

This theorem is proved by Lemma~\ref{singular fiber} and Corollary~\ref{final corollary}.

\medskip

{\bf Acknowledgements.}
Both authors would like to thank Jun-Muk Hwang to explain us the results of his papers and helpful comments, and would like to thank Arnaud Beauville for his interest and useful comments.

\bigskip

 \section{Lagrangian fibration structure on the contangent bundle}  


In order to show Theorem 1.1, we first describe the members of $H^0(X, \Sym^2T_{X})$ in terms of local parameters of $T_X^*$.

\begin{notation}\label{n.local}  Consider $\mathbb P^2=\mathbb P^2_{x_0,x_1,x_2}$ which means that $[x_0,x_1,x_2]$ is a  homogeneous coordinate system  of $\mathbb P^2$.  Set $x=\frac{x_1}{x_0}$ and $y=\frac{x_2}{x_0}$. Let $U_0=\mathbb A^2_{x,y}\subset \mathbb P^2$ be the affine open subset defined by $x_0\neq 0$. Take one $H\in H^0(U_0, \Sym^2T_{\mathbb P^2})$. Then we can write $H$ uniquely as 
	\begin{equation}\label{eq.local equation}H=f(x,y)\left(\frac{\partial}{\partial x}\right)^2+g(x,y)\left(\frac{\partial}{\partial y}\right)^2+h(x,y)\left(\frac{\partial}{\partial x}\right)\left(\frac{\partial}{\partial y}\right) \end{equation}
	where \begin{center}
		$f(x,y)=\sum_{i,j}f_{i,j}x^iy^j$, \  $g(x,y)=\sum_{i,j}g_{i,j}x^iy^j$, \  $h(x,y)=\sum_{i,j}h_{i,j}x^iy^j$ $\in\mathbb C[x,y]$. 
	\end{center}
\end{notation}

\subsection{Description of $H^0(\PP^2, \Sym^2T_{\PP^2})$}
\begin{lemma}\label{l.projective plane}In the situation of Notaion \ref{n.local},  
	$H$ is in $H^0(\mathbb P^2, \Sym^2T_{\mathbb P^2})$ if and only if  $\deg f$, $\deg g$, $\deg h\leq 4$, and the following 18 linear forms of the coefficients of $H$ vanish:

{	$$h_{0,4},\ h_{1,3}-2g_{0,4},\ h_{2,2}-2g_{1,3},\  h_{3,1}-2g_{2,2},\  h_{4,0}-2g_{3,1},\  g_{4,0}, $$ 
	$$f_{0,3},\ f_{1,2}-h_{0,3},\ f_{2,1}+g_{0,3}-h_{1,2},\  f_{3,0}+g_{1,2}-h_{2,1},\ g_{2,1}-h_{3,0},\  g_{3,0},$$  $$  f_{0,4},\ f_{1,3},\ f_{2,2}-g_{0,4},\ f_{3,1}-g_{1,3},\  f_{4,0}-g_{2,2},\  g_{3,1}.$$}
\end{lemma}
\begin{proof}
	Assume that $H\in H^0(\mathbb P^2, \Sym^2T_{\mathbb P^2}) $. 
	
	Let $u=\frac{x_0}{x_2}$ and $v=\frac{x_1}{x_2}$. Then $u$ and $v$ form an affine coordinate system on the affine open subset $U_2=\mathbb A^2_{u,v}\subset\mathbb P^2$ given by $x_2\neq 0$.  Since 
	$x=\frac{v}{u}$, $y=\frac{1}{u}$, $u=\frac{1}{y}$, $v=\frac{x}{y}$,
	we have  
	$$\frac{\partial}{\partial x}=\frac{\partial u}{\partial x}\frac{\partial}{\partial u}+\frac{\partial v}{\partial x}\frac{\partial}{\partial v}=\frac{1}{y}\frac{\partial}{\partial v}=u\frac{\partial}{\partial v} $$
	and $$\frac{\partial}{\partial y}=\frac{\partial u}{\partial y}\frac{\partial}{\partial u}+\frac{\partial v}{\partial y} \frac{\partial}{\partial v}=-\frac{1}{y^2}\frac{\partial}{\partial u}-\frac{x}{y^2}\frac{\partial}{\partial v}=-u^2\frac{\partial}{\partial u}-uv\frac{\partial}{\partial v}.$$
	Therefore  
	\begin{align*}H|_{U_2}&=f(\frac{v}{u},\frac{1}{u})u^2\left(\frac{\partial}{\partial v}\right)^2+g(\frac{v}{u},\frac{1}{u})\left(u^2\frac{\partial}{\partial u}+uv\frac{\partial}{\partial v}\right)\\&\hspace{3cm}-h(\frac{v}{u},\frac{1}{u})\left(u^3\left(\frac{\partial}{\partial u}\right)\left(\frac{\partial}{\partial v}\right)+u^2v\left(\frac{\partial}{\partial v}\right)^2\right)\\
		&=g(\frac{v}{u},\frac{1}{u})u^4\left(\frac{\partial}{\partial u}\right)^2+\left\{f(\frac{v}{u},\frac{1}{u})u^2+g(\frac{v}{u},\frac{1}{u})u^2v^2-h(\frac{v}{u},\frac{1}{u})u^2v\right\}\left(\frac{\partial}{\partial v}\right)^2
		\\&\hspace{3cm}+\left\{ 2g(\frac{v}{u},\frac{1}{u})u^3v- h(\frac{v}{u},\frac{1}{u}) u^3 \right\}\left(\frac{\partial}{\partial u}\right)\left(\frac{\partial}{\partial v}\right)
	\end{align*}	
	Since $H|_{U_2}\in H^0(U_2, \Sym^2T_{\mathbb P^2}) $,  the coefficients of 
	\begin{center}$\left(\frac{\partial}{\partial u}\right)^2$, $\left(\frac{\partial}{\partial v}\right)^2$, and $\left(\frac{\partial}{\partial u}\right)\left(\frac{\partial}{\partial v}\right)$\end{center}
 appearing in $H|_{U_2}$ above  are holomorphic functions in $u,v$. 
	Therefore $$g(\frac{v}{u},\frac{1}{u})u^4$$ which is the coefficient of $\left(\frac{\partial}{\partial u}\right)^2$ in $H|_{U_2}$ is holomorphic, and hence  
	\begin{center}
		$g_{i,j}=0$ for all $i,j$ with $i+j\geq 5$
	\end{center}
	 which means that  $\deg g\leq 4$. 
	 
	We have the following equalities:
	\begin{align*}
		&2g(\frac{v}{u},\frac{1}{u})u^3v- h(\frac{v}{u},\frac{1}{u}) u^3
		\\&=2\sum g_{i,j}(\frac{v}{u})^i(\frac{1}{u})^ju^3v-\sum h_{i,j}(\frac{v}{u})^i(\frac{1}{u})^j u^3
		\\&=2\sum g_{i,j}u^{3-i-j}v^{i+1}-\sum  h_{i,j}u^{3-i-j}v^i  \hspace{4cm}(*)
	\end{align*}
	From the same arguments as above, we can see that 
	the  coefficients of $u^{-l}v^k$ for $l>0$ and $k\geq 0$ in $(*)$   vanish which implies
	
	\begin{center}
		$h_{i,j}=0$ for all $i,j$ with $i+j\geq 5$ 
	\end{center}
	and 
	\begin{center}$h_{0,4}=0$, $h_{1,3}=2g_{0,4}$, $h_{2,2}=2g_{1,3}$, $h_{3,1}=2g_{2,2}$, $h_{4,0}=2g_{3,1}$, $g_{4,0}=0$. 
	\end{center}
	We have   the  following equalities: 
	\begin{align*}
		&f(\frac{v}{u},\frac{1}{u})u^2+g(\frac{v}{u},\frac{1}{u})u^2v^2-h(\frac{v}{u},\frac{1}{u})u^2v\\&=\sum f_{i,j}(\frac{v}{u})^i(\frac{1}{u})^ju^2+\sum g_{i,j}(\frac{v}{u})^i(\frac{1}{u})^ju^2v^2-\sum h_{i,j}(\frac{v}{u})^i(\frac{1}{u})^j u^2v\\
		&=\sum f_{i,j}u^{2-i-j}v^i+\sum g_{i,j}u^{2-i-j}v^{i+2}-\sum h_{i,j}u^{2-i-j}v^{i+1} \hspace{1cm}(**)
	\end{align*} By the same reason as before,  it follows that 
	the coefficients of $u^{-l}v^k$ with $l>0$ and $k\geq 0$ in $(**)$  vanish. This shows  that 
	
	\begin{center}
		$f_{i,j}=0$ for all $i+j\geq 5$
	\end{center}
	and 
	\begin{center}
		$f_{0,3}=f_{1,2}-h_{0,3}=f_{2,1}+g_{0,3}-h_{1,2}=f_{3,0}+g_{1,2}-h_{2,1}=g_{2,1}-h_{3,0}=g_{3,0}=0$, 
	\end{center}
	
	\begin{center}
		$f_{0,4}=f_{1,3}-h_{0.4}=f_{2,2}+g_{0,4}-h_{1,3}=f_{3,1}+g_{1,3}-h_{2,2}=f_{4,0}+g_{2,2}-h_{3,1}=g_{3,1}-h_{4,0}=g_{4,0}=0$.
	\end{center}
So we obtained  all the 18 linear relations in our lemma. 

	Coversely if $H$ satisfies the conditions in this lemma, then $H$ is holomorphic on  $\mathbb P^2\setminus \{(0:0:1)\}$ and hence it can be holomorphically extended to all $\mathbb P^2$. 
\end{proof}

\begin{remark} The above 18 linear relations  in Lemma \ref{l.projective plane} are independent and thus 
	$h^0(\mathbb P^2, \Sym^2T_{\mathbb P^2})=27$. The dimension can be also computed from the Euler sequence of $T_{\PP^2}$. From the exact sequence
	$$0\to\cO_{\PP^2}\to \cO_{\PP^2}(1)^{\oplus 3}\to T_{\PP^2}\to 0,$$
	we have
	$$0\to\cO_{\PP^2}(1)^{\oplus 3}\to {\rm Sym}^2(\cO_{\PP^2}(1)^{\oplus 3})\to {\rm Sym}^2T_{\PP^2}\to 0.$$

	We also remark that 
	Lemma \ref{l.projective plane} is equivalent to the following statement: $H\in H^0(\mathbb P^2, \Sym^2T_{\mathbb P^2})$ if and only if $\deg f, \deg g,\deg h\leq 4$, and 
	{\small\begin{align*}
		\frac{1}{x^2}f_{4}(x,y)&=\frac{1}{y^2}g_4(x,y)\\
		f_4(x,y)&=x^2(f_{4,0}x^2+f_{3,1}xy+f_{2,2}y^2)=\frac{x^2}{y^2}g_4(x,y)
		\\g_4(x,y)&=y^2(g_{0,4}y^2+g_{1,3}xy+g_{2,2}x^2)=\frac{y^2}{x^2}f_4(x,y)\\
		h_4(x,y)&=xy(h_{1,3}y^2+h_{2,2}xy+h_{3,1}x^2y)=\frac{2x}{y}g_4(x,y)=\frac{2y}{x}f_4(x,y)=\frac{y}{x}f_4(x,y)+\frac{x}{y}g_4(x,y).
	\end{align*}
	and 
	\begin{align*}
		f_3(x,y)&=x(f_{3,0}x^2+f_{2,1}xy+f_{1,2}y^2)\\
		g_3(x,y)&=y(g_{2,1}x^2+g_{1,2}xy+g_{0,3}y^2)\\
		h_3(x,y)&=h_{3,0}x^3+h_{2,1}x^2y+h_{1,2}xy^2+h_{0,3}y^3\\
		&=g_{2,1}x^3+(f_{3,0}+g_{1,2})x^2y+(f_{1,2}+g_{0,3})xy^2+f_{1,2}y^3\\
		&=\frac{y}{x}f_3(x,y)+\frac{x}{y}g_3(x,y).
	\end{align*}}
\end{remark}

\medskip

\subsection{Description of $H^0(X, \Sym^2T_{X})$}
Let $\mu_p: Y={\rm Bl}_p\mathbb P^2\rightarrow \mathbb P^2=\mathbb P^2_{x_0,x_1,x_2}$ be the blow-up at $p=(1:a:b)$. Then 

$$H^0(Y, \Sym^2T_Y)\subset H^0(\mathbb P^2, \Sym^2T_{\mathbb P^2})$$

Take  $H\in H^0(\mathbb P^2, \Sym^2T_{\mathbb P^2})$.  For the affine open subset  $U_0\subset \mathbb P^2$ defined by $x_0\neq 0$, write $H|_{U_0}$ as in Notaion \ref{n.local} so that it satisfies the properties in Lemma \ref{l.projective plane}.  

\begin{lemma}\label{l.blow-up}The above $H$ is in  $H^0(Y, \Sym^2T_Y)$ 
	if and only if 
	
	$$f(a,b)=g(a,b)=h(a,b)=g_x(a,b)=f_y(a,b)=0$$ and  $$g_y(a,b)-h_x(a,b)=f_x(a,b)-h_y(a,b)=0.$$
	
	
\end{lemma}

\begin{proof}
	For simplicity we only prove that  if $p=(1:0:0)$, then $H\in H^0(Y, \Sym^2T_Y)$ if and only if  
	\begin{center}
		$f_{0,0}=g_{0,0}=h_{0,0}=g_{1,0}=f_{0,1}=g_{0,1}-h_{1,0}=f_{1,0}-h_{0,1}=0$
	\end{center}The proof for the general case can be done by the same argument. 
	
Let $E$ be the exceptional divisor of $\mu_p$. 	Let $U_0\subset\PP^2=\PP^2_{x_0, x_1, x_2}$ be the affine open neighborhood  of $p$ defined by $x_0\neq 0$. 	
Consider $x=\frac{x_1}{x_0}$ and $y=\frac{x_1}{x_0}$ as an affine coordinate system on $U_0=\mathbb A^2_{x,y}\subset \mathbb P^2$ so that $p=(0,0)$. 
	Then $\mu_p^{-1}(U_0)\subset \mathbb A^2_{x,y}\times \mathbb P^1_{z_0,z_1}$ is defined by $xz_1=yz_0$.  Let $W\subset \mu_p^{-1}(U_0)$ be the open subset given by $z_0\neq 0$.  Set $w=\frac{z_1}{z_0}$ and $r=x$. Then $y=xw$, $W= \mathbb A^2_{r,w}$, and $E\cap W$ is defined by $r=0$ in $W$.  
	
	From the relations
	$$ \frac{\partial}{\partial x}= \frac{\partial r}{\partial x}\frac{\partial}{\partial r}+\frac{\partial w }{\partial x}\frac{\partial}{\partial w}= \frac{\partial}{\partial r}-\frac{y}{x^2}\frac{\partial}{\partial w}= \frac{\partial}{\partial r}-\frac{w}{r}\frac{\partial}{\partial w}  ,$$ and 
	$$\frac{\partial}{\partial y}= \frac{\partial r}{\partial y}\frac{\partial}{\partial r}+\frac{\partial w}{\partial y}\frac{\partial}{\partial w}= \frac{1}{r}\frac{\partial}{\partial w}$$  
	it follows that 
	\begin{align*}
		H|_{W\setminus E}&=f(r,rw)\left(\frac{\partial}{\partial r}-\frac{w}{r}\frac{\partial}{\partial w}\right)^2+g(r,rw)\left(\frac{1}{r}\frac{\partial}{\partial w}\right)^2+h(r,rw)\left(\frac{\partial}{\partial x}-\frac{w}{r}\frac{\partial}{\partial w}\right)\left(\frac{1}{r}\frac{\partial}{\partial w}\right)\\
		&=f(r,rw)\left(\frac{\partial}{\partial r}\right)^2
		+\left\{f(r,rw)\frac{w^2}{r^2}+g(r,rw)\frac{1}{r^2}-h(r,rw)\frac{w}{r^2}   \right\}\left(\frac{\partial}{\partial w}\right)^2
		\\&+\left\{ -2f(r,rw)\frac{w}{r}+h(r,rw)\frac{1}{r}\right\}\left(\frac{\partial}{\partial r}\right)\left(\frac{\partial}{\partial w}\right)
	\end{align*}
	The coefficient of $\left(\frac{\partial}{\partial w}\right)^2$ in $H|_{W\setminus E}$ satisfies  the following equalities:
	\begin{align*}
		&f(r,xw)\frac{w^2}{x^2}+g(r,rw)\frac{1}{r^2}-h(r,rw)\frac{w}{r^2}
		\\&=\sum f_{i,j}r^i(rw)^j\frac{w^2}{r^2}+\sum g_{i,j}r^i(rw)^j\frac{1}{r^2}
		-\sum h_{i,j}r^i(rw)^j\frac{w}{r^2}\\
		&=\sum f_{i,j}r^{i+j-2}w^{j+2}+\sum g_{i,j}r^{i+j-2}w^j-\sum h_{i,j}r^{i+j-2}w^{j+1}
	\end{align*}
	Therefore if $H|_{W\setminus E}$ holomorphically extends to $W$, 
	\begin{center}
		$f_{0,0}=g_{0,0}=h_{0,0}=0$ and  $g_{1,0}=g_{0,1}-h_{1,0}=f_{1,0}-h_{0,1}=f_{0,1}=0$
	\end{center} because $E\cap W$ is defined by $r=0$ in $W$. 
	
	Similarly, from the following equalities 
	\begin{align*}
		&-2f(r,rw)\frac{w}{r}+h(r,rw)\frac{1}{r}
		\\&=-2\sum f_{i,j}r^i(rw)^j\frac{w}{r}+\sum h_{i,j}r^i(rw)^j\frac{1}{r}
		\\&=-2\sum f_{i,j}r^{i+j-1}w^{j+1}+\sum h_{i,j}r^{i+j-1}w^j
	\end{align*}
	it follows that if $H|_{W\setminus E}$ holomorphically extends to $W$ then  $f_{0,0}=h_{0,0}=0$.

	Let $W'\subset \mu_p^{-1}(U)$ be the open subset defined by $z_1\neq 0$.  Let $w=\frac{z_0}{z_1}$.  Then $x=yw$ and $W'=\mathbb A^2_{y,w}$.
	Using the following relations
	$$ \frac{\partial}{\partial x}= \frac{\partial w}{\partial x}\frac{\partial}{\partial w}+\frac{\partial y }{\partial x}\frac{\partial}{\partial y}= \frac{1}{y}\frac{\partial}{\partial w},$$ and 
	$$ \frac{\partial}{\partial y}= \frac{\partial w}{\partial y}\frac{\partial}{\partial w}+\frac{\partial y}{\partial y}\frac{\partial}{\partial y}= -\frac{x}{y^2}\frac{\partial}{\partial w}+\frac{\partial}{\partial y}=-\frac{w}{y}\frac{\partial}{\partial w}+\frac{\partial}{\partial y},$$  we have 
	\begin{align*}
		H|_{W'\setminus E}&=f(yw,y)\left(\frac{1}{y}\frac{\partial}{\partial w}\right )^2+g(yw,y)\left(-\frac{w}{y}\frac{\partial}{\partial w}+\frac{\partial}{\partial y}\right)^2\\&\hspace{2cm}+h(yw,y)\left(\frac{1}{y}\frac{\partial}{\partial w}\right)\left(-\frac{w}{y}\frac{\partial}{\partial w}+\frac{\partial}{\partial y}\right)\\
		&=\left\{ f(yw,y)\frac{1}{y^2}+g(yw,y)\frac{w^2}{y^2}-h(yw,y)\frac{w}{y^2}\right\}\left(\frac{\partial}{\partial w}\right)^2
		+g(yw,y)  \left(\frac{\partial}{\partial y}\right)^2
		\\&\hspace{2cm}+\left\{ -2g(yw,y)\frac{w}{y}+h(yw,y)\frac{1}{y}\right\}\left(\frac{\partial}{\partial y}\right)\left(\frac{\partial}{\partial w}\right).
	\end{align*}
	The coefficient of $\left(\frac{\partial}{\partial w}\right)^2$ in 	$H|_{W'\setminus E}$ satisfies the following equalities:
	\begin{align*}
		&f(yw,y)\frac{1}{y^2}+g(yw,y)\frac{w^2}{y^2}-h(yw,y)\frac{w}{y^2}
		\\&=\sum f_{i,j}(yw)^iy^j\frac{1}{y^2}\sum g_{i,j}(yw)^iy^j\frac{w^2}{y^2}-\sum h_{i,j}(yw)^iy^j\frac{w}{y^2}\\
		&=\sum f_{i,j} y^{i+j-2} w^i+\sum_{i,j}g_{i,j}y^{i+j-2}w^{i+2}-\sum h_{i,j}y^{i+j-2}w^{i+1}
	\end{align*} 
     and thus if  $H|_{W'\setminus E}$ holomorphically extends to $W'$ then 
	\begin{center}
		$f_{0,0}=g_{0,0}=h_{0,0}=0$ and $f_{0,1}=f_{1,0}-h_{0,1}=g_{0,1}-h_{1,0}=g_{1.0}=0.$
	\end{center}
\end{proof}


\begin{lemma}\label{l.coordinate setting}
Let $p_1,\ldots,p_5$ be five distinct points in $\mathbb P^2$ in general position, i.e., no three of them lie in a line. Then we can choose a homogeneous coordinate system on $\mathbb P^2$  so that  
$p_1=(1:0:0)$, $p_2=(1:1:0)$, $p_3=(1:0:1)$, $p_4=(1:1:-1)$ or $(1:1:-1/2)$, and  $p_5=(1:a:b)$ for some $a,b\in\mathbb C$. 
\end{lemma}

\begin{proof}
Clearly we can choose a homogeneous coordinate system $x_0,x_1,x_2$ on $\mathbb P^2$  so that  
\begin{center}$p_1=(1:0:0)$, $p_2=(1:1:0)$, $p_3=(1:0:1)$ and $p_4=(1:1:-1)$.\end{center}
Assume that $p_5=(0:1:b)$. Set 
$$M=\begin{bmatrix}
x& y &z\\
0 & x+y& 0\\
0&0& x+z
\end{bmatrix}.
$$ 
Let us change the homogeneous coordinate system $x_0,x_1,x_2$  by the linear transform on $\mathbb P^2$ given by a matrix of the from $M$ above such that 
$$y+bz=x+y=-b(x+z)\neq 0.$$
In this new coordinates, we have 
\begin{center}$p_1=(1:0:0)$, $p_2=(1:1:0)$, $p_3=(1:0:1)$,   $p_5=(1:1:-1)$\end{center}
and $$p_4=(x+y-z:x+y:-(x+z)).$$
Assume that   $x+y-z=0$. Then 
$$x+y-z=bz-(b^2+b)z-z=-(b^2+b+1)z=0.$$
Since $z\neq 0$ we have 
$$b^2+b+1=0.$$

Let us change the initial homogeneous  coordinate system $x_0,x_1,x_2$ on $\mathbb P^2$ by a matrix of the form  $M$  above such that $$y+bz=x+y=-2b(x+z)\neq 0.$$
In this new coordinates, we have 
\begin{center}
$p_1=(1:0:0)$, $p_2=(1:1:0)$, $p_3=(1:0:1)$, $p_5=(1:1:-1/2)$     
\end{center}
and $$p_4=(x+y-z:x+y:-(x+z)).$$  Then
$$x+y-z=bz-(2b^2+3b)z-z=-(2b^2+2b+1)z\neq 0$$ because $b^2+b+1=0$ and $z\neq 0$. 
We get our lemma.
\end{proof}

\medskip
\subsection{Tangents of Lagrangian fibration}

Let   $\mu:X=\rm{Bl}_{p_1,\ldots,p_5}\mathbb P^2\rightarrow \mathbb P^2$ be the blow up at  five   points $p_1,\ldots,p_5\in\mathbb P^2$  in general position. Let $E_i\subset X$ be the exceptional curve over $p_i$.

Take any two independent sections $H,G\in  H^0(X, \Sym^2T_X)$. Then we can consider $H$ and $G$ as regular funtions on $T_X^*$ so that they define a morphism   
 \begin{center}$\Phi:T_X^*\rightarrow \mathbb C^2$, $q\mapsto (H(q),G(q))$.\end{center} 
We recall that for each $e\in\mathbb C^2$, the fiber  $S_e=\Phi^{-1}(e)$ is called a level surface.  
	
 By Lemma \ref{l.coordinate setting}, we can  choose a homogeneous coordinate system  $x_0,x_1,x_2$ on $\mathbb P^2$ so that  
\begin{center}
    $p_1=(1:0:0)$, $p_2=(1:1:0)$, $p_3=(1:0:1)$, $p_4=(1:\alpha:\beta)$ ,  and  $p_5=(1:a:b)$ 
\end{center}
for some $a,b\in\mathbb C$, and $(\alpha,\beta)=(1,-1)$ or $(1,-1/2)$.

Let $U_0\subset \mathbb P^2=\mathbb P^2_{x_0,x_1,x_2}$ be the affine open subset defined by $x_0\neq 0$. Set $x=\frac{x_1}{x_0}$, $y=\frac{x_2}{x_0}$, 
$u=\frac{\partial}{\partial x}$ and $v=\frac{\partial}{\partial y}$. 
Let $U:=\Pi^{-1}(\mu^{-1}(U_0)\setminus \cup_{i=1}^5 E_i)$. Here $\Pi:T^*_X\rightarrow X$ is the projection morphism.    We can consider the restrictions  $H|_{U}$ and $G|_{U}$  as members in $\mathbb C[x,y,u,v]$.

\bigskip
Let us consider the canonical symplectic two form $\omega$ on $T_X^*$: $\omega|_{U}$ can be expressed as
$$ \omega|_{U}=dx\wedge du+dy\wedge dv.$$

\begin{lemma}\label{l.Lagrangian condition}Take  $e\in\mathbb C^2$ and  $q\in S_e\cap U$. If $\dim T_qS_e=2$ and $$H_y(q)G_v(q)-H_v(q)G_y(q)+H_x(q)G_u(q)-H_u(q)G_x(q)=0$$ then 
	$\omega|_{T_qS_e}=0$. (Here $H_x(q)=\frac{\partial{H}}{\partial x}(q)$, $H_y(q)=\frac{\partial{H}}{\partial y}(q)$, and so on.) 
\end{lemma}

\begin{proof}
We remark that the tangent space $T_q(S_e)\subset T_q(T_X^*)$ is defined by 
\begin{align}\label{eq.tangent eq 1}
	d_qH&=H_x(q)d_qx+H_y(q)d_qy+H_u(q)d_qu+H_v(q)d_qv=0 \mbox{ and}\\ 
\label{eq.tangent eq 2} d_qG&=G_x(q)d_qx+G_y(q)d_qy+G_u(q)d_qu+G_v(q)d_qv=0. \end{align}

	Assume that  $\dim T_q S_e=2$ and $$H_y(q)G_v(q)-H_v(q)G_y(q)+H_x(q)G_u(q)-H_u(q)G_x(q)=0.$$ Then $$A:=-H_u(q)\frac{\partial}{\partial x}|_q-H_v(q)\frac{\partial}{\partial y}|_q+H_x(q)\frac{\partial}{\partial u}|_q+H_y(q)\frac{\partial}{\partial v}|_q$$
	and $$B:=G_u(q)\frac{\partial}{\partial x}|_q+G_v(q)\frac{\partial}{\partial y}|_q-G_x(q)\frac{\partial}{\partial u}|_q-G_y(q)\frac{\partial}{\partial v}|_q$$ are tangent  vectors in $T_q(S_e)$ because they satisfy the two equations  (\ref{eq.tangent eq 1}) and (\ref{eq.tangent eq 2}). Since $\dim T_q(S_e)=2$, they also form  a basis of $T_q(S_e)$. By our assumption we have 
	$$\omega(A,B)=H_y(q)G_v(q)-H_v(q)G_y(q)+H_x(q)G_u(q)-H_u(q)G_x(q)=0$$ which implies that $\omega|_{T_q(S_e)}=0$. 
\end{proof}

\begin{proposition}\label{p.lagrangian}
For all $q\in S_e^{sm}\cap U$ we have 
$\omega|_{T_q(S_e)}=0.$ 
 Here $S_{e}^{sm}$ denotes the smooth locus of $S_{e}$. 
\end{proposition}

\begin{proof}
By Lemma \ref{l.Lagrangian condition} it is enough to show that  for all $q\in S_{e}^{sm}\cap U$,  \begin{center}$H_y(q)G_v(q)-H_v(q)G_y(q)+H_x(q)G_u(q)-H_u(q)G_x(q)=0$. \end{center}

Let us  write $H|_{U}$ and $G|_{U}$ as in  Notation \ref{n.local}  so that 
$$H|_{U}=H(x,y,u,v)=f(x,y)u^2+g(x,y)v^2+h(x,y)uv\in\mathbb C[x,y,u,v]$$ 
and 
$$G|_{U}=G(x,y,u,v)=c(x,y)u^2+d(x,y)v^2+e(x,y)uv\in\mathbb C[x,y,u,v]$$ where 
$$f(x,y)=\sum_{i+j\leq 4}f_{i,j}x^iy^j,\  g=\sum_{i+j\leq 4}g_{i,j}x^iy^j,\  h=\sum_{i+j\leq 4}h_{i,j}x^iy^j$$  and 
$$c(x,y)=\sum_{i+j\leq 4}c_{i,j}x^iy^j,\  d=\sum_{i+j\leq 4}d_{i,j}x^iy^j,\  e=\sum_{i+j\leq 4}e_{i,j}x^iy^j$$ are polynomials in variables $x$ and $y$ 
satisfying  the conditions in  Lemmas \ref{l.projective plane} and \ref{l.blow-up}.

	Set 
 	\begin{align*}R:&=G_v(x,y,u,v)H_y(x,y,u,v)-H_v(x,y,u,v)G_y(x,y,u,v)\\&\hspace{1cm}+H_x(x,y,u,v)G_u(x,y,u,v)-H_u(x,y,u,v)G_x(x,y,u,v) .
	\end{align*} 
	Then $R$ is a member of the polynomial ring $$P:=\mathbb C[f_{i,j}, g_{i,j}, h_{i, j}, c_{i,j}, d_{i,j}, e_{i,j},x,y,u,v,a,b\ |\ i+j\leq 4].$$ Let $I$ be the ideal of $P$ generated by following polynomials given by 
	the conditions in Lemmas \ref{l.projective plane} and \ref{l.blow-up}.
	
{\small	$$h_{0,4}, h_{1,3}-2g_{0,4}, h_{2,2}-2g_{1,3}, h_{3,1}-2g_{2,2}, h_{4,0}-2g_{3,1}, g_{4,0}, $$ 
	$$f_{0,3}, f_{1,2}-h_{0,3}, f_{2,1}+g_{0,3}-h_{1,2}, f_{3,0}+g_{1,2}-h_{2,1}, g_{2,1}-h_{3,0}, g_{3,0},$$  $$  f_{0,4}, f_{1,3}, f_{2,2}-g_{0,4}, f_{3,1}-g_{1,3}, f_{4,0}-g_{2,2}, g_{3,1},$$
	
	$$e_{0,4}, e_{1,3}-2d_{0,4}, e_{2,2}-2d_{1,3}, e_{3,1}-2d_{2,2}, e_{4,0}-2d_{3,1}, d_{4,0}, $$ 
	$$c_{0,3}, c_{1,2}-e_{0,3}, c_{2,1}+d_{0,3}-e_{1,2}, c_{3,0}+d_{1,2}-e_{2,1}, d_{2,1}-e_{3,0}, d_{3,0},$$  $$  c_{0,4}, c_{1,3}, c_{2,2}-d_{0,4}, c_{3,1}-d_{1,3}, c_{4,0}-d_{2,2}, d_{3,1},$$

	$$f(0,0), g(0,0), h(0,0), g_x(0,0), f_y(0,0), g_y(0,0)-h_x(0,0), f_x(0,0)-h_y(0,0),$$
	$$f(1,0), g(1,0), h(1,0), g_x(1,0), f_y(1,0),g_y(1,0)-h_x(1,0), f_x(1,0)-h_y(1,0),$$
	$$f(0,1), g(0,1), h(0,1), g_x(0,1), f_y(0,1), g_y(0,1)-h_x(0,1), f_x(0,1)-h_y(0,1),$$
	$$f(\alpha,\beta), g(\alpha,\beta), h(\alpha,\beta), g_x(\alpha,\beta), f_y(\alpha,\beta), g_y(\alpha,\beta)-h_x(\alpha,\beta), f_x(\alpha,\beta)-h_y(\alpha,\beta),$$
	$$f(a,b), g(a,b), h(a,b), g_x(a,b), f_y(a,b), g_y(a,b)-h_x(a,b), f_x(a,b)-h_y(a,b), {\rm and}$$
	
	$$c(0,0), d(0,0), e(0,0), d_x(0,0), c_y(0,0), d_y(0,0)-e_x(0,0), c_x(0,0)-e_y(0,0),$$
	$$c(1,0), d(1,0), e(1,0), d_x(1,0), c_y(1,0),d_y(1,0)-e_x(1,0), c_x(1,0)-e_y(1,0),$$
	$$c(0,1), d(0,1), e(0,1), d_x(0,1), c_y(0,1), d_y(0,1)-e_x(0,1), c_x(0,1)-e_y(0,1),$$
	$$c(\alpha,\beta), d(\alpha,\beta), e(\alpha,\beta), d_x(\alpha,\beta), c_y(\alpha,\beta), d_y(\alpha,\beta)-e_x(\alpha,\beta), c_x(\alpha,\beta)-e_y(\alpha,\beta),$$
	$$c(a,b), d(a,b), e(a,b), d_x(a,b), c_y(a,b), d_y(a,b)-e_x(a,b), c_x(a,b)-e_y(a,b).$$
}	

\medskip
		For the proof it is enough to show  the following claim.  
		
		\medskip
		Claim: If we fix $a$ and $b$ so that $p_5=(1:a:b)\neq p_i$ for all $i=1,\ldots,4$, then 
		$R$ vanishes if $H$ and $G$ satisfy the relations in Lemmas \ref{l.projective plane} and \ref{l.blow-up}. 	
		\medskip

	By using Magma calculator, we can show that 
	
{\small	\begin{center} $(a-1)abR$, 
		$(a-1)(b+1)aR$, 
		$(a-1)abR$, 
		$(b-\beta)abR$, and
		$(\frac{1-\beta}{\alpha}a +b-1))bR$	\end{center}}\noindent are members in $I$.

		Assume that $a\neq 0$ and $b\neq 0$. Since $p_5\neq p_4$, we have  $a\neq 1$ or $b\neq \beta$ which implies that the claim because $(a-1)abR, (b-\beta)abR\in I$.

		Assume that $a\neq 0$ and $b=0$. Then $a\neq 1$ since  $p_5\neq p_2$. This implies the claim because $(a-1)(b+1)aR\in I$. 
		
		Assume that $a=0$.  Since $p_5\neq p_1, p_3$, we can see  that $b\neq 0,1$. So we get the claim  because $(\frac{1-\beta}{\alpha}a+b-1)bR\in I $. 
\end{proof}

\medskip
\begin{proof}[\bf Proof of Theorem \ref{t.lagrangian}]
Let $X$ be a del Pezzo surface of degree 4. By Lemma \ref{l.coordinate setting}, we may assume that $X$ is the blow-up of $\mathbb P^2_{x_0,x_1,x_2}$ at five distinct  points  $p_1=(1:0:0)$, $p_2=(1:1:0)$, $p_3=(1:0:1)$,  $p_4=(1:\alpha:\beta)$, and $p_5=(1:a:b)$  for some $a,b\in\mathbb C$, and $(\alpha, \beta)=(1,-1)$ or $(1,-1/2)$. 
Since the restriction $\Pi|_{S_e}: S_e\rightarrow X$ is surjective for all $e\in\mathbb C^2$,   $S^{sm}_e\cap U$ forms a dense open subset of $S_e^{sm}$. From this and Proposition \ref{p.lagrangian} we get the theorem. 
\end{proof}

\medskip
\begin{remark}\label{relation} The question on the relation between the Lagrangian fibration structure of the Hitchin map for the case of $g=2$ and the Lagrangian fibration structure of the cotangent bundle of a del Pezzo surface $X$ of degree 4 was raised by Beaville and Brambila-Paz when the second named author gave a talk at the the conference for Fabrizio Catanese's 70th birthday. Let $Z=SU^s_C(2, 1)$ where $C$ is a smooth projective curve of genus 2. By the Hitchin map $h_Z: T^*_Z\to\CC^3=H^0(C, 2K_C)$,
$$\bigoplus_{m=0}^{\infty} H^0(Z, {\rm Sym}^m T_Z)\simeq\CC[F_1, F_2, F_3]$$
where $F_i\in H^0(Z, {\rm Sym}^2 T_Z)$. This is an isomorphism of graded rings. It is well known that $Z$ is a complete intersection of two smooth quadrics $\bar Q_1$ and $\bar Q_2$ in $\PP^5$. More precisely, if a genus two curve $C$ is defined by six Weierstrass points $\lambda_i$ for $i=1,\ldots, 6$ then $Z$ is isomorphic to the complete intersection of two quadrics (cf. \cite{NR}, \cite{Newstead}, \cite{DR})
$$\bar Q_1=\sum_{i=1}^6 X_i^2=0, \quad \bar Q_2=\sum_{i=1}^6 \lambda_iX_i^2=0.$$

The above question is whether we can find a $Z$ such that each fiber $\Phi^{-1}(e)$ of the Lagrangian fibration of a del Pezzo surface $X$ can be embedded naturally into each fiber of the Hitchin map $h_Z: T^*_Z\to\CC^3=H^0(C, 2K_C)$ with $X=Z\cap H$ where $H$ is a hyperplane in $\PP^5$. 

There is a natural identification of the pencil $\mathbb P^1_Q$ of quadrics induced by $\bar Q_1$ and $\bar Q_2$ with $\mathbb P(H^0(C, K_C))$. Also from Theorem 5.1 in \cite{OL19}, there is an also isomorphism of graded rings:
$$\bigoplus_{m=0}^{\infty} H^0(Z, \Sym^m[\Omega^1_Z(1)])\simeq \mathbb C[\bar Q_1,\bar Q_2].$$
Then the preimage $h^{-1}_Z(W)$ of the Hitchin map of the image of a natural embedding $W:={\rm Sym}^2H^0(C, K_C)$ in $\CC^3=H^0(C, 2K_C)$ is the locus of singular spectral curves. This identification is explained in detail in the thesis of Sarbeswar Pal \cite{Pal}. Also recently, Hitchin \cite{Hit21} studies explicitly the Hitchin map $h_Z: T^*_Z\to\CC^3=H^0(C, 2K_C)$.

Then the question is whether the restriction of this $h^{-1}_Z(W)$ over $X$, which is the intersection of $Z$ with some hyperplane section $H$, is the cotangent bundle of $X$. Since we have a natural identification between 
$$\bigoplus_{m=0}^{\infty} H^0(Z\cap H, \Sym^m[\Omega^1_{Z\cap H}(1)])\simeq\CC[\bar Q_1\cap H, \bar Q_2\cap H]\quad {\rm and}$$  
$$\bigoplus_{m=0}^{\infty} H^0(Z, \Sym^m[\Omega^1_Z(1)])\simeq \mathbb C[\bar Q_1,\bar Q_2],$$
and due to the description of an irreducible component of a general fiber of the locus of singular spectral curves (Theorem 1.3 in \cite{HO09}), if the question is true then a general fiber of the locus of singular spectral curves seems to be isomorphic either a $\PP^1$ bundle or an elliptic fiber bundle over $S_e$ (up to \'etale cover) in Theorem 1.2. We cannot answer on this question now because there is no enough study on the $h_Z: T^*_Z\to\CC^3=H^0(C, 2K_C)$. We leave it for the future study.
\end{remark}

\bigskip

\section{Level surfaces in the Lagrangian fibration}

We use the same notations as in the introduction. Let $X$ be a del Pezzo surface of degree 4. 
Let $Q_1$ and $Q_2$ be two quadratic forms in variables $y_1,\ldots,y_5$ defining $X\subset\mathbb P^4=\mathbb P^4_{y_1,\ldots,y_5}$ such that  $\det Q_1=1$. We define the characteristic polynomial $P(t):=\det (tQ_1-Q_2)$, then it satisfies 
$$P(t)=\prod_{i=1}^5(t-\theta_i)$$  
where all $\theta_i\in\mathbb C$ are distinct.

We have a pencil of quadric hypersurfaces in $\mathbb P^4$:
$$\psi:\mathcal Q=\{Q_{\bold e}\}_{\bold e\in\mathbb P^1_{e_1,e_2}}\rightarrow \mathbb P^1_{e_1,e_2}=\mathbb P^1$$
such that its fiber at  $\be=(e_1:e_2)\in\mathbb P^1_{e_1,e_2}$ corresponds to the quadric hypersurface $Q_{\be}$ in $\mathbb P^4$ defined by  $e_2Q_1-e_1Q_2=0$.

For each  $i=1,\ldots,5$, set  $\bold a_i=(1:\theta_i)\in \mathbb P^1_{e_1,e_2}$. Then  $Q_{\be}$ is singular exactly only when $\be=\bold a_i$ for some $i$.

 \begin{lemma}[\cite{Sk}]\label{l.del Pezzo blow up}
 $X$ is isomorphic to the blow-up of $\mathbb P^2$ at the images $p_i\in\mathbb P^2$ of $\bold a_i\in \PP^1_{e_1,e_2}$  under the Veronese embedding $\mathbb P^1\hookrightarrow\mathbb P^2$, and is isomorphic to the subscheme of  $\mathbb P^4_{y_1,\ldots,y_5}$  defined by
\begin{equation}\label{eq.quadrics2}\sum_{i=1}^5P'(\theta_i)^{-1}y_i^2=\sum_{i=1}^5P'(\theta_i)^{-1}\theta_iy_i^2=0.\end{equation}
\end{lemma}

\medskip

\subsection{Surfaces in the linear system $|2\zeta|$ in $\PP(T_X)$}
 
Let us consider $X$ as the blow-up of $\mathbb P^2$ at $p_1,\ldots,p_5\in \mathbb P^2$ in Lemma \ref{l.del Pezzo blow up} and denote by $\mu:X\rightarrow \mathbb P^2$ the blow-up morphism.

\subsubsection{Description of lines in $X$} Let $E_i$ be  the exceptional curve on $X$ over $p_i$ and  $C$ the proper transform  of the unique   conic in $\mathbb P^2$ through all $p_1,\ldots,p_5$.  For each $1\leq i\neq j\leq 5$, let $\ell_{i,j}\subset X$ be the proper transform of the line in $\mathbb P^2$ connecting $p_i$ and $p_j$.    Then $C$, $\{E_i\}_i$ and $\{\ell_{i,j}\}_{i,j}$ are exactly the 16 lines $\ell_1,\ldots,\ell_{16}$  in $X$  in the introduction.  We  denote by ${C'}$, ${E_i'}$ and ${\ell_{i,j}'}\subset \mathbb P(T_X)$ the  sections  of the respective lines associated  quotients of the form: $T_X|_{\mathbb P^1}=\mathcal O_{\mathbb P^1}(2)\oplus \mathcal O_{\mathbb P^1}(-1)\twoheadrightarrow \mathcal O_{\mathbb P^1}(-1)$. 

\subsubsection{ Linear system $|2\zeta|$}As seen in the introduction, the pencil $\{Q_{\bold e}\}_{\bold e\in\mathbb P^1_{e_1,e_2}}$ of quadric hypersurfaces induced by  $Q_1$ and $Q_2$  gives the linear system $|2\zeta|$ in $\PP(T_X)$. Let $\ell_i'$ be 16 sections of $\mathbb P(T_X|_{\ell_i})\rightarrow \ell_i$ which are associated to  quotients 
$T_X|_{\ell_i}=\mathcal O_{\mathbb P^1}(2)\oplus \mathcal O_{\mathbb P^1}(-1)\twoheadrightarrow \mathcal O_{\mathbb P^1}(-1)$.   Let  $B$ be the base locus  of the linear system $|2\zeta|$ in $\PP(T_X)$. Using $\zeta\cdot \ell_i'=-1$ and the description of Section 2 in \cite{OL19} we can show that 
$B$ is supported on  the disjoint union of 16 sections $\ell_i'$ so that   $B=\sum_{i=1}^{16}{a_i\ell_i'}$ for some integers $a_i\geq 1$. 
Using the Grothendieck relation 
$$\zeta^2+\pi^*K_X\cdot \zeta++\pi^*c_2(T_X)=0$$
we can  calculate  $\zeta^3=-4$. Therefore
$$(2\zeta)^2\cdot \zeta=\sum_{i=1}^{16}a_i\ell_i'\cdot \zeta=-\sum_{i=1}^{16}a_i=-16$$
which implies that $a_i=1$ for all $i$.

Let $$\mu_B:{\rm Bl}_B\mathbb P(T_X)\rightarrow\mathbb P(T_X)$$ be the bolow-up along  $B$. 
We have  a smooth member $K_g$ of $|2\zeta|$ (see Corollary 2.4 in \cite{DH}).
The exact sequence on normal bundles 
$$0\rightarrow N_{\ell_i'/K_g}=\mathcal O_{\ell_i'}(-2)\rightarrow N_{\ell_i'/\mathbb P(T_X)}\rightarrow N_{K_g/\mathbb P(T_X)}|_{\ell_i'}=\mathcal O_{\ell_i'}(-2)\rightarrow 0$$ 
shows that 
$N_{\ell_i'/\mathbb P(T_X)}\cong\mathcal O_{\mathbb P^1}(-2)\oplus\mathcal O_{\mathbb P^1}(-2) $  because ${\rm Ext}^{1}(\mathcal O_{\mathbb P^1}(-2),\mathcal O_{\mathbb P^1}(-2))=0$. 
Thus the exceptional divisor over $\ell_i'$ of the blow-up $\mu_B$  is isomorphic to $\mathbb P(\mathcal O_{\mathbb P^1}(-2)\oplus \mathcal O_{\mathbb P^1}(-2))\cong\mathbb P^1\times \mathbb P^1$. This implies that, after blow-up, the rational map $\tilde\phi:\mathbb P(T_X)\dashrightarrow \mathbb P^1_{e_1,e_2}$ induced by the morphism $\Phi:T_X^*\rightarrow \mathbb C^2_{e_1,e_2}$ defined by the pair $(Q_1,Q_2)$  can extended to a 
 morphism $$\phi:{\rm Bl}_B\mathbb P(T_X)\rightarrow\mathbb P^1_{e_1,e_2}=\mathbb P^1$$ which is a family of members of $|2\zeta|$. 
We often consider each fiber $K_{\bold e}=\phi^{-1}(\bold e)$  as  a subscheme of $\mathbb P(T_X)$. 

For each ${\bold e}\in \mathbb P^1_{e_1,e_2}$, we let $$\pi_{\be}:K_{\be}\rightarrow X$$ be the restriction of the composition  $\pi\circ\mu_B:{\rm Bl}_B\mathbb P(T_X)\rightarrow\mathbb P(T_X)\rightarrow X$. 
\begin{lemma}\label{l.double cover}
For each point $x$ in $X$, $\pi_{\be}^{-1}(x)$ consists of two points with multiplicity except only when $Q_{\be}$ is singular and $x$ is one of the intersection points of some two lines in $X$. In this exceptional case, $\pi_{\be}^{-1}(x)$ is isomorphic to   $\mathbb P^1$.  
\end{lemma}
 \begin{proof}Let $Q$ be a smooth quadric hypersurface in $\mathbb P^4$ such that $X=Q_{\be}\cap Q$.  The description of Section 2 in \cite{OL19} says that 
 each fiber $\pi_{\be}^{-1}(x)$ parametrizes lines in $Q_\be\cap {\bold T}_x X$ through $x$, where ${\bold T}_x X\subset \PP^4$ denotes the embedded projective tangent plane to $X$ at $x$. So  we only need to show the next claim. 
 
 \medskip
 Claim: For a point $x$ in $ X$, ${\bf T}_xX\cap Q_{\be}={\bf T}_xX$ if and only if $Q_{\be}$ is singular and $x$ is the intersection point of some two lines in $X$. 

\medskip
Assume that ${\bf T}_xX\cap Q_{\be}={\bf T}_xX$. Then ${\bf T}_xX\subset Q_{\be}$. Since any smooth quadric hypersurface in $\mathbb P^4$  contains no plane in $\mathbb P^4$,   $Q_{\be}$ is singular so that it is a cone over a quadric surface  in $\mathbb P^3$. We also have equalities ${\bf T}_xX\cap X= {\bf T}_xX\cap Q_{\be}\cap Q={\bf T}_xX\cap Q$ as a set, which implies  that ${\bf T}_xX\cap X$ is a union of some two lines in $X$ and $x$ is the intersection point of them.

Conversely, assume that $Q_{\be}$ is singular and $x\in X$ is the intersection point of some  two lines $\ell_{\iota_1}$ and $\ell_{\iota_2}$ in $X$.  Then the intersection $Q_{\be}\cap {\bf T}_xX$ contains $\ell_{\iota_1}$, $\ell_{\iota_2}$ and some other line in the ruling of the cone structure on $Q_{\be}$ which implies that  $Q_{\be}\cap {\bf T}_xX={\bf T}_xX$.
 \end{proof}

\subsubsection{Conic fibration on $X$} Let  ${\rm RatCurves}^n(X)$ be the normalized space of rational curves on $X$ (see \cite{Kol96}). 
For each $i=1,\ldots,5$, let $\mathcal K_{i,1}$ be the irreducible component of  ${\rm RatCurves}^n(X)$  containing the proper transform of a general line in $\mathbb P^2$  through $p_i$, and let $\mathcal K_{i,2}$ be that containing the proper transform of a general conic in $\PP^2$  through $\{p_1,\ldots,p_5\}\setminus \{p_i\}$.  
 There is a conic fibration  $$\pi_{i,j}:X\rightarrow \mathbb P^1$$ whose general fiber is a member of $\mathcal K_{i,j}$. 
The conic fibration $\pi_{i,j}:X\rightarrow  \PP^1$ has four singular fibers.  For each $k\in\{1,\ldots,5\}\setminus \{i\}$, there is a  singular fiber of $\pi_{i,1}$ which  is the union of $\ell_{k,i}$ and $E_k$. Three of the four singular fibers of $\pi_{i,2}$ are the unions of two lines  of the forms $\ell_{\iota_1,\iota_2}$ and $\ell_{\iota_3,\iota_4}$ with $\{{\iota_1},{\iota_2},{\iota_3},{\iota_4}\}=\{1,\ldots,5\}\setminus \{i\}$, and the last one is the union of $C$ and  $E_i$. We note that the union of singular fibers of $\pi_{i,1}$ and $\pi_{i,2}$ is exactly the union of 16 lines $\ell_1,\ldots,\ell_{16}$  in $X$.

\subsubsection{Fibration on Total dual VMRT }  
 Let $\breve{\mathcal C}_{i,j}$ be  the total dual VMRT  associated to 
$\mathcal K_{i,j}$. 
We refer to the paper \cite{HLS} for the total dual VMRT. 
 Let $L$ be the class of $\mu^*\mathcal O_{\mathbb P^2}(1)$.
By Corollary 2.13 in \cite{HLS}, we have 
 \begin{align*}[\breve{\mathcal C}_{i,1}]&=\zeta-\pi^*L+\pi^*[E_1]+\cdots+\pi^*[E_5]-2\pi^*[E_i], \mbox{ and} \\ 
[\breve{\mathcal C}_{i,2}]&=\zeta+\pi^*L-\pi^*[E_1]-\cdots-\pi^*[E_5]+2\pi^*[E_i]\end{align*}
and thus 
\begin{equation}\label{eq.reducible surface}[\breve{\mathcal C}_{i,1}]+[\breve{\mathcal C}_{i,2}]=2\zeta. \end{equation} 
 This shows that there are 5 points  $\bold b_1,\ldots,\bold b_5$  in $\mathbb P^1_{e_1,e_2}$ such that $K_{\bold b_i}=\breve{\mathcal C}_{i,1}\cup \breve{\mathcal C}_{i,2}$. 
\begin{lemma}\label{l.reducible and singular}
 The  5 points $\bold b_i\in \mathbb P^1_{e_1,e_2}$ are the same as the 5 points $\bold a_i\in\mathbb P^1_{e_1,e_2}$ after reordering. 
\end{lemma}

\begin{proof}
We only need to show that each $Q_{\bold b_i}$ is singular. Suppose not. Then by Lemma~\ref{l.double cover}, $\breve{\mathcal C}_{i,1}$ and $\breve{\mathcal C}_{i,2}$ are isomorphic to $X$. Clearly at least one of $\breve{\mathcal C}_{i,1}$ and $\breve{\mathcal C}_{i,2}$ contains some $\ell_k'$.  

Assume that $\breve{\mathcal C}_{i,1}$ contains some $\ell_k'$. Let $F$ be a general fiber of $\pi_{i,1}:X\rightarrow \mathbb P^1$. Then $F\cdot\zeta|_{{\breve{\mathcal C}}_{i,1}}=0$ and  $\ell_k'\cdot\zeta|_{\breve{\mathcal C}_{i,1}}=-1$. Since $F$ is a conic and $\ell_k$ is a line in $X$, we have  $[F]=2[\ell_k']$ in $\breve{\mathcal C}_{i,1}$,  a contradiction.  
\end{proof}

We have a fibration  $\breve\pi_{i,j}:\breve{\mathcal C}_{i,j}\rightarrow \mathbb P^1$  of curves on $\breve{\mathcal C}_{i,j}$   given by the composition   
 $$\breve\pi_{i,j}:=\pi_{i,j}\circ \pi|_{\breve{\mathcal C}_{i,j}}:\breve{\mathcal C}_{i,j}\rightarrow X\rightarrow\mathbb P^1.$$  

\begin{lemma}\label{l.VMRT fibration}
The restriction $\pi|_{\breve{\mathcal C}_{i,j}}:\breve{\mathcal C}_{i,j}\rightarrow X$ of $\pi: \mathbb P(T_X)\rightarrow X$ is  the blow-up of four points of $X$  which are the singular points of the singular fibers of the conic fibration $\pi_{i,j}:X\rightarrow\mathbb P^1 $.  Moreover the proper transform in $\breve{\mathcal C_{i,j}}$ of a line  $\ell_k$  in a singular fiber of $\pi_{i,j}$ is equal to $\ell'_k$.
\end{lemma}
\begin{proof}
Let us  assume that  a singular fiber of $\pi_{i,1}:X\rightarrow \mathbb P^1$ consists of two lines $\ell_{\iota_1}$ and $\ell_{\iota_2}$ meeting at $x$. Then by Lemmas \ref{l.double cover} and \ref{l.reducible and singular} the preimge $\pi_{\bold b_i}^{-1}(x)$  of $\pi_{\bold b_i}:K_{\bold b_i}\rightarrow X$ is isomorphic to $\mathbb P^1$. Since $\pi_{\bold b_i}|_{\breve{\mathcal C}_{i,1}}=\pi|_{\breve{\mathcal C}_{i,1}}$  and $\pi_{\bold b_i}|_{\breve{\mathcal C}_{i,2}}=\pi|_{\breve{\mathcal C}_{i,2}}$ give isomorphisms  in the outside of singular fibers of $\pi_{i,1}$ and $\pi_{i,2}$ respectively,  the fiber   $\pi_{\bold b_i}^{-1}(x)$ is contained in $\breve{\mathcal C}_{i,1}$. This shows the first statement in our lemma. 

The fiber of $\breve\pi_{i,1}:\breve{\mathcal C}_{i,1}\rightarrow X\rightarrow \mathbb P^1$ over  the singular fiber $\ell_{\iota_1}\cup\ell_{\iota_2}$ of $\pi_{i,1}$ consists of the  proper transforms
$\hat{\ell_{\iota_1}}$ and $\hat{\ell_{\iota_2}}$ 
of  $\ell_{\iota_1}$ and $\ell_{\iota_2}$ respectively, and  $2\ell$ where $\ell$ is the exceptional curve over $x$. Clearly $\ell$ is a fiber of $\pi:\mathbb P(X)\rightarrow X$ and hence $\ell\cdot \zeta|_{\breve{\mathcal C}_{i,1}}=1$.  From this and  $(\mbox{fiber of } \breve{\pi}_{i,1})\cdot \zeta|_{\breve{\mathcal C}_{i,1}}=0$, it follows that 
$\hat{\ell_{\iota_1}}\cdot \zeta|_{\breve{\mathcal C}_{i,1}}=-1$ and $\hat{\ell_{\iota_2}}\cdot \zeta|_{\breve{\mathcal C}_{i,1}}=-1$. This implies that  $\hat{\ell_{\iota_1}}=\ell_{\iota_1}'$ and $\hat{\ell_{\iota_2}}=\ell_{\iota_2}'$. We are done. 
\end{proof}


We know that  the five $K_{\bold b_i}$ are  reducible. Next lemma shows that there is no other reducible  $K_{\be}$.  

\begin{lemma}\label{l.characterization of irreducible fibers}
 For any  $\be\in\mathbb P^1\setminus\{\bold b_1,\ldots,\bold b_5\}$, $K_{\be}$ is irreducible and the morphism $\pi_{\be}:K_{\be}\rightarrow X$ is a double cover, i.e., a finite morphism of degree 2. 
\end{lemma}
\begin{proof}
 Since $Q_{\be}$ is smooth,  Lemma \ref{l.double cover} says that the morphism $\pi_{\be}$ is a finite morphism of degree 2. 
If $K_{\be}$ is reducible, then it is a union of two irreducible components which are isomorphic to $X$. We have a contradiction by the  same reason as in the proof of Lemma~\ref{l.reducible and singular}.
\end{proof}

For each $\be \in \mathbb P^1\setminus \{\bold b_1,\ldots,\bold b_5\}$, 
let $D_{\be}\subset X$ be the branch curve of the double covering $\pi_{\be}:K_{\be}\rightarrow X$ (see Lemma \ref{l.characterization of irreducible fibers});
 the branch curve $D_{\be}$ is the locus of points $x$ such that $Q_\be\cap {\bold T}_x X$ is a double line. 

\subsubsection{ General $K_{\be}$}
For a general ${\bold e}\in \mathbb P^1_{e_1,e_2}\setminus\{\bold b_1,\ldots, \bold b_5\}$, 
$K_\be$ is a K3 surface of degree 8 of Kummer type; Since $K_{\PP(T_X)}=-2\zeta$, $K_{K_\be}=\cO_{K_\be}$. So the branch curve $D_\be\subset X$ of the double covering $\pi_{\bold e}=\pi|_{K_{\bold e}}:K_{\bold e}\rightarrow X$ is in $|\cO_X(2)|$ because $-K_X=\cO_X(1)$ where $\cO_X(1)$ is a hyperplane section of $X$ in $\PP^4$. Therefore $D_\be$ is a nonsingular curve of genus 5 with degree 8 in $\PP^4$, and tangent to all 16 lines $\ell_i$ in $X$. And the lifts $\ell_i'\subset\mathbb P(T_X)$ of the  16 lines $\ell_i$ in $X$ as in the introduction are   $(-2)$-curves $\ell_{\be,i}$ in $K_\be\subset{\rm Bl}_B\mathbb P(T_X) $. These 16 (-2)-curves $\ell_{\be,i}$ are the intersection of  $K_\be$ with the exceptional divisor $D$ of $\mu_B:{\rm Bl}_B\mathbb P(T_X)\rightarrow \mathbb P(T_X)$. By the blow-down $\mu:X\rightarrow \PP^2$, $D_\be$ goes to a plane sextic curve with five cusps at $\{ p_1, \ldots, p_5\}$.

By the above explanation, we obtain the following lemma. 

\begin{lemma}\label{l.general K} For a general ${\bold e}\in \mathbb P^1_{e_1,e_2}\setminus\{\bold b_1,\ldots,\bold b_5\}$, 
$K_\be$ is a K3 surface of degree 8 of Kummer type. It has 16 (-2)-curves $\ell_{\bold e,i}$  which are intersection of  $K_{\bold e}$ with the exceptional divisor $D$ of the blow-up $\mu_B:{\rm Bl}_B\mathbb P(T_X)\rightarrow \mathbb P(T_X)$.  
\end{lemma}

\begin{remark} We know
\[ \chi_{\rm top}({\rm Bl}_B \PP(T_X)) =\chi_{\rm top}(\PP(T_X))+32= \chi_{\rm top}(X)\cdot 2+32 = 48.\]
Since $\chi_{\rm top}(\text{K3 surface})=24$ and $\chi_{\rm top}(K_{\bold b_i})=24$ for all $i=1, \ldots, 5$, $\chi_{\rm top}(K_\be)=24$ for a general ${\bold e}\in\mathbb P^1_{e_1,e_2}$ and $\{\bold b_1, \ldots, \bold b_5\} \subset \mathbb P^1_{e_1,e_2}$. This seems to imply that for every ${\bold e}\in \mathbb P^1_{e_1,e_2}\setminus\{\bold b_1,\ldots,\bold b_5\}$, $K_\be$ is a K3 surface of degree 8 of Kummer type. In Corollary~\ref{final corollary}, we prove that this is true by considering on the Lagragian fibration of the map $T_X^*\to\CC^2_{e_1,e_2}$. 
\end{remark}

\begin{remark} By the result by Skorobogatov (Theorem~3.1 in \cite{Sk}), we have more explicit description of $K_\be$ for a general ${\bold e}\in \mathbb P^1_{e_1,e_2}\setminus\{\bold b_1,\ldots,\bold b_5\}$. There exists an embedding 
 $K_{\bold e}\subset \mathbb P^5_{y_1,\ldots,y_6}$ so that it is defined by 
$$\sum_{i=1}^6Q'(\theta_i)^{-1}y_i^2=\sum_{i=1}^6Q'(\theta_i)^{-1}\theta_iy_i^2
=\sum_{i=1}^6Q'(\theta_i)^{-1}\theta_i^2y_i^2=0$$
where $\theta_1,\ldots,\theta_5$ are the same $\theta_i$s in Lemma~\ref{l.del Pezzo blow up},  $\theta_6$ is determined by $K_\be$, and 
$$Q(t):=\prod_{i=1}^6(t-\theta_i).$$ 
Furthermore the restriction of the projection map 
 \begin{center}$\mathbb P^5_{y_1,\ldots,y_6}\dashrightarrow \mathbb P^4_{y_1,\ldots,y_5}$,   
 $(y_1:\cdots:y_6)\mapsto (y_1:\ldots:y_5)$\end{center} gives a double cover 
 $\pi_{\bold e}:K_{\bold e}\rightarrow X$ 
 branched on a degree 8 curve $D_{\bold e}$ in $X$ defined by 
 $$ \sum_{i=1}^5Q'(\theta_i)^{-1}y_i^2=\sum_{i=1}^5Q'(\theta_i)^{-1}\theta_iy_i^2
=\sum_{i=1}^5Q'(\theta_i)^{-1}\theta_i^2y_i^2=0.$$
\end{remark}

\subsubsection{ Reducible $K_{\be}$}

When $K_\be$ goes to $K_{\bold b_i}=\breve{\mathcal C}_{i,1}\cup \breve{\mathcal C}_{i,2}$, the branch curve $D_\be$ of $\pi_\be:K_{\be}\rightarrow X$ goes to 2$E_{\bold b_i}$ where $ E_{\bold b_i}$ is an elliptic curve which is a hyperplane section of $X$ in $\PP^4$. The image of $E_{\bold b_i}$ of the blow-up $\mu:X\rightarrow \mathbb P^2$ is a cubic curve   in $\mathbb P^2$  tangent to  the line $\ell_{i,k}$ at $p_k$ for each  $k\in\{1,\ldots,5\}\setminus \{i\}$. For each $i$, this cubic plane curve is uniquely determined by this property. 

We can observe  that $E_{\bold b_i}$ is the closure of the locus of $x$ in $X$ such that  some two conics in $X$ which are members of  $\mathcal K_{i,1} $ and $\mathcal K_{i,2}$ respectively  tangentially intersect at $x$. 

\begin{lemma}\label{l.intersection of E and fibers}
$E_{\bold b_i}$ meet smooth fibers of $\pi_{i,j}:X\rightarrow \mathbb P^1$ at two  distinct points,  and  the singular fibers of it at the singular points. 
\end{lemma}
\begin{proof}
Given a smooth conic curve in $\mathbb P^2$ through four points $\{p_1,\ldots,p_5\}\setminus \{p_i\}$, there are two distinct lines in $\mathbb P^2$ through the point $p_i$ which are tangent lines of the given conic curve. When this smooth conic specializes to a singular conic in $X$, the above two distinct lines in $\mathbb P^2$ goes to the unique  double line in $\mathbb P^2$ through $p_i$ and  the singular point of that singular conic curve. This shows our lemma for $\pi_{i,2}$. The proof for the fibers of $\pi_{i,1}$ can be done in a similar method.  
 \end{proof}
 From Lemma~\ref{l.intersection of E and fibers}, it follows that the restriction $\pi_{i,j}|_{E_{\bold b_i}}:E_{\bold b_i}\rightarrow \mathbb P^1$ is a double cover branched at the four singular values of $\pi_{i,j}:X\rightarrow \mathbb P^1$. 


The intersection curve between two components $\breve{\mathcal C}_{i,1}$ and $\breve{\mathcal C}_{i,2}$ is the proper transform of $E_{\bold b_i}$ of the blow-up $\pi|_{\breve{\mathcal C}_{i,j}}:\breve{\mathcal C}_{i,j}\to X$. We will denote it by the same notation $E_{\bold b_i}$.   By Lemma~\ref{l.intersection of E and fibers}, $E_{\bold b_i}$ intersects two  distinct points at each smooth fiber of $\breve{\pi}_{i,j}:\breve{\mathcal C}_{i,j}\to\PP^1$, and one point with multiplicity two at the exceptional curve of singular fibers of it; We note that the multiplicity of this exceptional curve is two in a singular fiber.

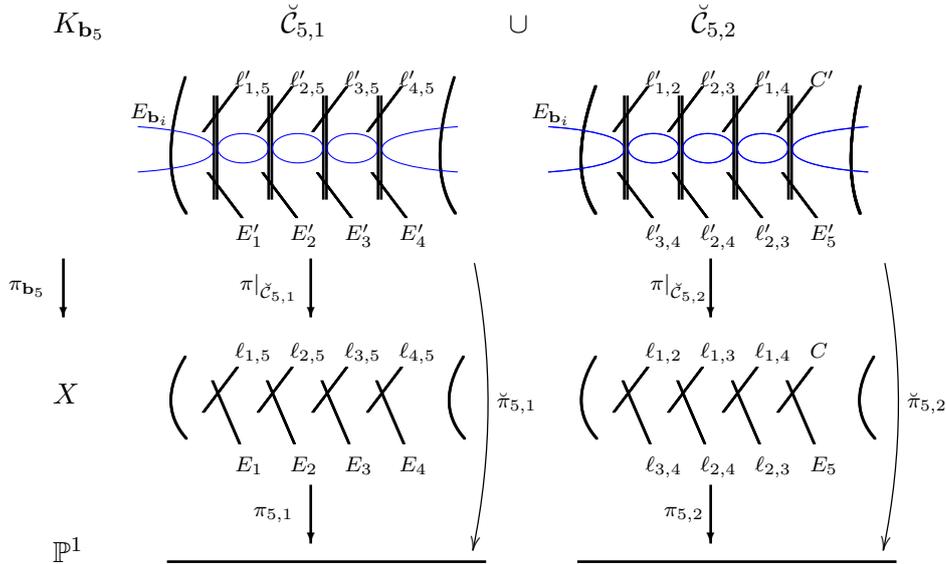
\unitlength 0.6mm 
\linethickness{1pt}
\ifx\plotpoint\undefined\newsavebox{\plotpoint}\fi 
\begin{picture}(100,160)(20,70)

	\put(104.25,110){ \vector(0,-1){13}}

	\put(50,128){\small $X$}
	\put(50,93){\small  $\mathbb P^1$}
	\put(94,103){\tiny$\pi_{5,1}$}
	\put(90,138){\tiny$\ell_{1,5}$}
	\put(102,138){\tiny $\ell_{2,5}$}
	\put(114,138){\tiny $\ell_{3,5}$}
	\put(126,138){\tiny $\ell_{4,5}$}
	\put(90,113){\tiny $E_1$}
	\put(102,113){\tiny $E_2$}
	\put(114,113){\tiny $E_3$}
	\put(126,113){\tiny $E_4$} 
	
	\qbezier(79,138.25)(72.625,127.75)(79.25,120.25)
	\qbezier(140,138.25)(133.625,127.75)(140.25,120.25)
	
	\put(75,93){\line(1,0){70}}
	\put(165,93){\line(1,0){70}}

	\put(194.25,110){\vector(0,-1){13}}
	\put(184,103){\tiny $\pi_{5,2}$}
	\put(180,138){\tiny $\ell_{1,2}$}
	\put(192,138){\tiny $\ell_{1,3}$}
	\put(204,138){\tiny $\ell_{1,4}$}
	\put(216,138){\tiny $C$}
	\put(180,113){\tiny $\ell_{3,4}$}
	\put(192,113){\tiny $\ell_{2,4}$}
	\put(204,113){\tiny $\ell_{2,3}$}
	\put(216,113){\tiny $E_5$} 
	
	\qbezier(169,138.25)(162.625,127.75)(169.25,120.25)
	\qbezier(230,138.25)(223.625,127.75)(230.25,120.25)


	\multiput(173,126)(.03,.04){250}{\line(0,1){.04}}
	\multiput(185,126)(.03,.04){250}{\line(0,1){.04}}
	\multiput(197,126)(.03,.04){250}{\line(0,1){.04}}
	\multiput(209,126)(.03,.04){250}{\line(0,1){.04}}
	\multiput(175,133)(.03,-.07){200}{\line(0,-1){.04}}
	\multiput(187,133)(.03,-.07){200}{\line(0,-1){.04}}
	\multiput(199,133)(.03,-.07){200}{\line(0,-1){.04}}
	\multiput(211,133)(.03,-.07){200}{\line(0,-1){.04}}
	
		\multiput(83,126)(.03,.04){250}{\line(0,1){.04}}
	\multiput(95,126)(.03,.04){250}{\line(0,1){.04}}
	\multiput(107,126)(.03,.04){250}{\line(0,1){.04}}
	\multiput(119,126)(.03,.04){250}{\line(0,1){.04}}
	\multiput(85,133)(.03,-.07){200}{\line(0,-1){.04}}
	\multiput(97,133)(.03,-.07){200}{\line(0,-1){.04}}
	\multiput(109,133)(.03,-.07){200}{\line(0,-1){.04}}
	\multiput(121,133)(.03,-.07){200}{\line(0,-1){.04}}

\put(50,160){ \vector(0,-1){13}}
\put(38,153){ \tiny $\pi_{\bold b_5}$}

\put(104.25,160){ \vector(0,-1){13}}

\put(90,198){\tiny$\ell_{1,5}'$}
\put(102,198){\tiny $\ell_{2,5}'$}
\put(114,198){\tiny $\ell_{3,5}'$}
\put(126,198){\tiny $\ell_{4,5}'$}
\put(90,164){\tiny $E_1'$}
\put(102,164){\tiny $E_2'$}
\put(114,164){\tiny $E_3'$}
\put(126,164){\tiny $E_4'$} 

\qbezier(79,200)(72.625,180)(79.25,170)
\qbezier(138,200)(131.625,180)(138.25,170)

\put(194.25,160){\vector(0,-1){13}}

\put(180,198){\tiny$\ell_{1,2}'$}
\put(192,198){\tiny $\ell_{2,3}'$}
\put(204,198){\tiny $\ell_{1,4}'$}
\put(216,198){\tiny $C'$}
\put(180,164){\tiny $\ell_{3,4}'$}
\put(192,164){\tiny $\ell_{2,4}'$}
\put(204,164){\tiny $\ell_{2,3}'$}
\put(216,164){\tiny $E_5'$}

\put(100,210){\small $\breve{\mathcal C}_{5,1}$}
\put(190,210){\small $\breve{\mathcal C}_{5,2}$}

\put(50,210){\small $K_{\bold b_5}$} 
\put(150,210){\small $\cup$}
\put(67,191){\tiny $E_{\bold b_i}$}
\put(155,191){\tiny $E_{\bold b_i}$}

\put(91,153){\tiny$\pi|_{\breve{\mathcal C}_{5,1}}$}
\put(181,153){\tiny$\pi|_{\breve{\mathcal C}_{5,2}}$}


\put(140,160){\small$\xymatrix @R=9pc @C=10pc{\ar@/^/[d]^{\breve \pi_{5,1}} &\\ & \\ &  }$}
\put(230,160){\small$\xymatrix @R=9pc @C=10pc{\ar@/^/[d]^{\breve\pi_{5,2}} &\\ & \\ &  }$}


\multiput(83,188)(.03,.04){250}{\line(0,1){.04}}
\multiput(95,188)(.03,.04){250}{\line(0,1){.04}}
\multiput(107,188)(.03,.04){250}{\line(0,1){.04}}
\multiput(119,188)(.03,.04){250}{\line(0,1){.04}}
\multiput(84,179)(.03,-.04){250}{\line(0,-1){.04}}
\multiput(96,179)(.03,-.04){250}{\line(0,-1){.04}}
\multiput(108,179)(.03,-.04){250}{\line(0,-1){.04}}
\multiput(120,179)(.03,-.04){250}{\line(0,-1){.04}}
\put(86,196){\line(0,-1){23}}\put(85.3,196){\line(0,-1){23}}
\put(98,196){\line(0,-1){23}}\put(97.3,196){\line(0,-1){23}}
\put(110,196){\line(0,-1){23}}\put(109.3,196){\line(0,-1){23}}
\put(122,196){\line(0,-1){23}}\put(121.3,196){\line(0,-1){23}}

\qbezier(169,198)(162.625,180)(169,170)
\qbezier(227,198)(223,180)(227,170)

\put(176,181){\begin{tikzpicture}
\node[draw,blue,thin,ellipse,minimum height=5pt,minimum width=19pt](e){};
\end{tikzpicture}} 

\put(188,181){\begin{tikzpicture}
\node[draw,blue,thin,ellipse,minimum height=5pt,minimum width=19pt](e){};
\end{tikzpicture}} 

\put(200,181){\begin{tikzpicture}
\node[draw,blue,thin,ellipse,minimum height=5pt,minimum width=19pt](e){};
\end{tikzpicture}}

\put(212,179){\begin{tikzpicture}
    \draw[blue,thin] plot [smooth,tension=1.5] coordinates{(1,0.3) (0,0) (1,-0.3)};
\end{tikzpicture}}

\put(158.5,179){\begin{tikzpicture}
    \draw[blue,thin] plot [smooth,tension=1.5] coordinates{(0,0.3) (1,0) (0,-0.3)};
\end{tikzpicture}}


\multiput(173,188)(.03,.04){250}{\line(0,1){.04}}
\multiput(185,188)(.03,.04){250}{\line(0,1){.04}}
\multiput(197,188)(.03,.04){250}{\line(0,1){.04}}
\multiput(209,188)(.03,.04){250}{\line(0,1){.04}}
\multiput(174,179)(.03,-.04){250}{\line(0,-1){.04}}
\multiput(186,179)(.03,-.04){250}{\line(0,-1){.04}}
\multiput(198,179)(.03,-.04){250}{\line(0,-1){.04}}
\multiput(210,179)(.03,-.04){250}{\line(0,-1){.04}}
\put(176,196){\line(0,-1){23}}\put(175.3,196){\line(0,-1){23}}
\put(188,196){\line(0,-1){23}}\put(187.3,196){\line(0,-1){23}}
\put(200,196){\line(0,-1){23}}\put(199.3,196){\line(0,-1){23}}
\put(212,196){\line(0,-1){23}}\put(211.3,196){\line(0,-1){23}}

\qbezier(169,198)(162.625,180)(169,170)
\qbezier(227,198)(223,180)(227,170)

\put(176,181){\begin{tikzpicture}
\node[draw,blue,thin,ellipse,minimum height=5pt,minimum width=19pt](e){};
\end{tikzpicture}} 

\put(188,181){\begin{tikzpicture}
\node[draw,blue,thin,ellipse,minimum height=5pt,minimum width=19pt](e){};
\end{tikzpicture}} 

\put(200,181){\begin{tikzpicture}
\node[draw,blue,thin,ellipse,minimum height=5pt,minimum width=19pt](e){};
\end{tikzpicture}}

\put(212,179){\begin{tikzpicture}
    \draw[blue,thin] plot [smooth,tension=1.5] coordinates{(1,0.3) (0,0) (1,-0.3)};
\end{tikzpicture}}

\put(158.5,179){\begin{tikzpicture}
    \draw[blue,thin] plot [smooth,tension=1.5] coordinates{(0,0.3) (1,0) (0,-0.3)};
\end{tikzpicture}}

\put(86,181){\begin{tikzpicture}
\node[draw,blue,thin,ellipse,minimum height=5pt,minimum width=19pt](e){};
\end{tikzpicture}} 

\put(98,181){\begin{tikzpicture}
\node[draw,blue,thin,ellipse,minimum height=5pt,minimum width=19pt](e){};
\end{tikzpicture}} 

\put(110,181){\begin{tikzpicture}
\node[draw,blue,thin,ellipse,minimum height=5pt,minimum width=19pt](e){};
\end{tikzpicture}}

\put(122,179){\begin{tikzpicture}
    \draw[blue,thin] plot [smooth,tension=1.5] coordinates{(1,0.3) (0,0) (1,-0.3)};
\end{tikzpicture}}

\put(68.5,179){\begin{tikzpicture}
    \draw[blue,thin] plot [smooth,tension=1.5] coordinates{(0,0.3) (1,0) (0,-0.3)};
\end{tikzpicture}}

\end{picture}

So  we obtain the following  lemma.

\begin{lemma}\label{singular fiber} For each $i=1,\ldots,5$, we have the following description of $K_{\bold b_i}$.
\begin{itemize}
    \item[(i)] $K_{\bold b_i}$ consists two irreducible components $\breve{\mathcal C}_{i,1}$ and $\breve{\mathcal C}_{i,2}$.
    \item[(ii)] Each $\breve{\mathcal C}_{i,j}$ for $j=1, 2$ is isomorphic to the blow-up of four distinct points of $X$. These four points are singular points of the four singular fibers of the conic fibrartion $\pi_{i,j}:X\rightarrow \mathbb P^1$.
    \item[(iii)] $\breve{\mathcal C}_{i,1}\cap \breve{\mathcal C}_{i,2}$ is a smooth elliptic curve $E_{\bold b_i}$.
    \item[(iv)] In the fibration $\breve{\pi}_{i,j}:\breve{\mathcal C}_{i,j}\to\PP^1$, $E_{\bold b_i}$ intersects two distinct points at each smooth fiber, and one point at the exceptional curve of each singular fiber.
\end{itemize} 
\end{lemma}

\medskip

\subsection{Description of level surfaces $S_e$}
From now on, we want to describe level surfaces $S_e$. As seen in the introduction, $S_e$ is defined by $\Phi^{-1}(e)$ where $$ \Phi:T^*_X\rightarrow \mathbb C^2_{e_1,e_2}=\mathbb C^2$$
is the morphism defined by   $(Q_1,Q_2)$.  Here we consider $Q_i$ as  sections in $H^0(X, \Sym^2T_X)$.

Let $$ \Pi_{e}: S_e\rightarrow X $$ be the restriction of $\Pi:T_X^*\rightarrow X$.  
Take  $e\neq 0\in\mathbb C^2_{e_1,e_2}$ and denote by $\bold e\in \mathbb P^1$   the image point of $e$ under the quotient map  $\mathbb C^2_{e_1,e_2}\setminus\{0\}\rightarrow \mathbb P^1_{e_1,e_2}$. There is a morphism $$\tau_e:S_{e}\rightarrow K_\be$$ induced by the quotient map $T_X^*\dashrightarrow \mathbb P(T_X)$. So we have the following commutative diagram:

\begin{center}{
\begin{tikzcd}
S_e\arrow[dr,"\Pi_{e}"'] \arrow[rr, "\tau_e"]& & K_{\be}\arrow[ld,"\pi_{\bold e}"]\\
& X &
\end{tikzcd}}
\end{center}

Since there is a graded ring isomorphism:
$$\bigoplus_{m=0}^{\infty}H^0(X, \Sym^mT_X)\simeq \mathbb C[Q_1,Q_2],$$
$S_{\lambda e}\cong S_e$ for all $\lambda\in\CC^*$. We have the following diagram of maps.

\unitlength 0.5mm 
\linethickness{1pt}
\ifx\plotpoint\undefined\newsavebox{\plotpoint}\fi 
\begin{picture}(100,110)(20,100)
 \put(90,160)
{$\xymatrix @R=0.5pc @C=1pc {
{\rm Bl}_B\mathbb P(T_X)\ar[r]\ar[rddd]_{\phi}& \mathbb P(T_X)\ar[rd]_{\pi}\ar@{.>}^{\tilde \phi}[ddd]&  &T_X^*\ar@{.>}[ll]\ar[ddd]^{\Phi}\ar[ld]^{\Pi}\\
&  &X & \\
& & &\\
& \mathbb P^1_{e_1,e_2} & &\mathbb C^2_{e_1,e_2}\ar@{.>}[ll]
}$}
\put(145, 177){\tiny $K_{\bold e}$}
\put(147,170){\tiny $\cap$}
\put(146, 190){\tiny $\ell_{i}'$}
\put(147,183){\tiny $\cap$}

\put(178, 130){\tiny $\ell_i$ }
\put(178, 137){\tiny $\cup$}

\put(110,177){\tiny $K_{\bold e}$}
\put(112,170){\tiny $\cap$}

\put(110,190){\tiny $\ell_{\bold e,i}$}
\put(112,183){\tiny $\cap$}

\put(201,177){\tiny $S_e$}
\put(203,170){\tiny $\cap$}
\put(163,177){\tiny $\xymatrix@R=0.5pc @C=1pc{&&\ar[ll]^{\tau_e}}$}
\put(130,177){\tiny $\simeq$}
\put(130,190){\tiny $\simeq$}
\put(200, 110){\tiny $e $}
\put(163, 110){\tiny $\xymatrix@R=0.5pc @C=1pc{&&\ar@{|-{>}}[ll]}$}

\put(148, 110){\tiny $\bold e$}
\end{picture}

\begin{lemma} For every $e\in \CC^2_{e_1,e_2}\setminus\{ (0)\}$, there is an involution $\imath$ on $S_e$ acting freely and the morphism $\tau_{e}:S_e\rightarrow K_\be$ factors through the quotient map $S\rightarrow S/\imath$, i.e., $$\tau_e:S_e\rightarrow S_e/\iota\hookrightarrow K_{\be}$$
 so that 
$S_{e}/\imath=K_\be\setminus\cup_{i=1}^{16}\ell_{\be,i}$

\end{lemma} 
\begin{proof}
For each $e=(e_1,e_2)$ and  each point in $X$, there is an open  neighborhood $U\cong\mathbb C^2_{x,y}$  of that point such that   $S_{e}|_{\Pi^{-1}(U)}\subset \Pi^{-1}(U)\cong\mathbb C^4_{x,y,u,v}$ is locally defined by  equations  \begin{align*}Q_1&=f(x,y)u^2+g(x,y)v^2+h(x,y)uv=e_1
\mbox{ and  }\\
Q_2&=c(x,y)u^2+d(x,y)v^2+e(x,y)uv=e_2.\end{align*}
 Here $Q_1=H$ and $Q_2=G$ in the notations in  Section 2.3. 
So for a general point in $X$, there are four points in the preimage of the map $\Pi_e: S_e\to X$. And there is a natural involution $\imath: (x,y,u, v) \mapsto (x,y,-u, -v)$ acting freely on $S_e$.

Let $t=\frac{u}{v}$. Then given $(x,y)$, the solution of the equation
$$e_2(ft^2+g+ht)=e_1(ct^2+d+et)$$
gives a fiber of the map $\pi_{\be}:K_\be\to X$ and a fiber of the map $S_e/\imath\to X$. Therefore 
$S_e/\imath\hookrightarrow K_{\be}$. Since the base locus of the linear system $|2\zeta|$ consists exactly of 16 sections  $\ell_{i}'$ which are $\ell_{\be,i}$ in $ K_{\be}$, we have  $S_{e}/\imath= K_\be\setminus\cup_{i=1}^{16}\ell_{\be,i}.$ \end{proof}

\subsubsection{General $S_e$} Take general $e\in \mathbb C^2_{e_1,e_2}$  so that $S_e$ is smooth. 
 The preimage $\pi_{\be}^{-1}(\ell_i)$ of  $\ell_i\subset X$ splits into two curves in $K_\be$, one is  $(-2)$-curves $\ell_{\be, i}$ and the the other is a  conic, denoted by $\tilde\ell_{\be, i}$,  cut by a trope (Remark 8.6.9 in \cite{Dol}). 
Let $\bar K_\be$ be a Kummer quartic surface with 16 nodes obtained by contracting 16 (-2)-curves $\ell_{\be,i}$ in $K_\be$. It is well known that $\bar K_\be$ has a double cover $\bar S_e$ which is an abelian surface. We note that $\bar S_e$ does not contain any rational curve because it is an abelian surface. The level surface $S_e$ for a general $e$ is $\bar S_e\setminus\{{\text{16 points}\}}$ where these 16 points are the preimage of 16 nodes of the double cover $\bar S_e\to\bar K_\be$. Next figure shows these situations.

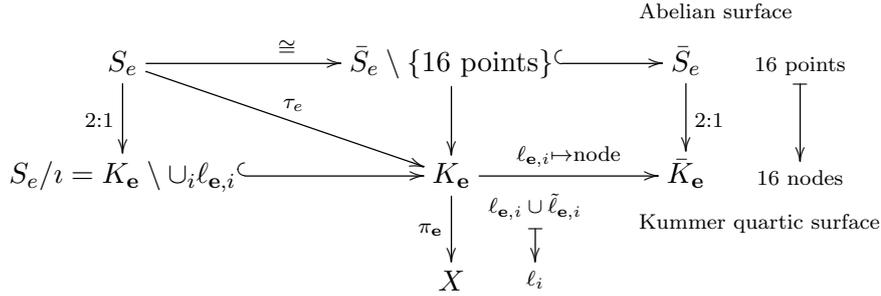
\unitlength 0.7mm 
\linethickness{1pt}
\ifx\plotpoint\undefined\newsavebox{\plotpoint}\fi 
\begin{picture}(100,80)(70,100)
 \put(90,160)
{\small $\xymatrix @R=2pc @C=3pc {
S_e\ar[r]^{\cong}\ar[d]_{\tiny 2:1}\ar[dr]^{\tau_e}&{\bar{S}_e}\setminus \{\mbox{16 points}\}\ar@{^{(}-{>}}[r]\ar[d]&\bar S_e \ar[d]^{\tiny 2:1}\\
S_e/\imath=K_{\be}\setminus \cup_i\ell_{\be,i}\ar@{^{(}-{>}}[r]& K_{\be}\ar[r]^{\tiny \ell_{\be,i}\mapsto \mbox{\tiny node}}\ar[d]_{\pi_{\be}} &\bar K_{\be}\\
&X &  
}
 $}

\put(180,133){\tiny $\xymatrix @R=1pc @C=1pc{ \ell_{\be,i}\cup \tilde{\ell}_{\be,i}\ar@{|-{>}}[d]\\
\ell_i}$}

\put(230,160){\tiny $\xymatrix @R=2.5pc @C=1pc{\mbox{16 points}\ar@{|-{>}}[d]\\ \mbox{16 nodes}} $}

\put(210,130){\tiny Kummer quartic surface}
\put(210,170){\tiny Abelian surface}

\end{picture}

\begin{remark}
Take $\bar\be\neq \bold b_i\in\mathbb P^1_{e_1,e_2}$.  Let $C_{\bar\be}$ be the smooth curve of genus 2 with 6 Weierstrass points over $\bold b_1,\ldots,\bold b_5,\bar\be\in\mathbb P^1_{e_1,e_2}$ under the hyperelliptic involution $C_{\bar\be}\rightarrow\mathbb P^1_{e_1,e_2}$.  If $K_\be$ is smooth then $S_{e}$ can be embedded into  the Jacobian variety $J_{\bar\be}$ of $C_{\bar\be}$ for some $\bar\be$ so that $J_{\bar\be}\setminus S_{e}$ consists of 16 disjoint points. It is not clear to us that $\bar\be =\be$.
\end{remark}


\bigskip

\subsubsection{ Reducible $S_e$} 

For each $i$,  we take one point $b_i\in\mathbb C^2_{e_1,e_2}$  over $\bold b_i$ under the quotient map $\mathbb C^2_{e_1,e_2}\setminus \{0\}\rightarrow \mathbb P^1_{e_1,e_2}$. We recall that $S_{b_i}\cong S_{\lambda b_i}$ for all $\lambda\in \CC^*$.  
Now let us describe $S_{b_i}$ by using the explicit description of $K_{\bold b_i}$ in Lemma~\ref{singular fiber}.  Let  $A_{i, j}$ be the preimage $\tau_{b_i}^{-1}(\breve{\mathcal C}_{i,j})$ so that $S_{b_i}=A_{i,1}\cup A_{i,2}$. 
The restriction $\tau_{i,j}=\tau_{b_i}|_{A_{i,j}}:A_{i,j}\rightarrow \breve{\mathcal C}_{i,j}$ is a finite morphism of degree 2 whose image is equal to $\breve{\mathcal C}_{i,j}\setminus \cup_{k=1}^{16} \ell_{\be,k}$.  We remark that each  $\breve{\mathcal C}_{i,j}$ contains only 8 (-2) curves and  the multiplicity of the exceptional curves of $\pi|_{\breve{\mathcal C}_{i,j}}: \breve{\mathcal C}_{i,j}\rightarrow X$ is two in the singular fiber of $\breve{\pi}_{i,j}: \breve{\mathcal C}_{i,j}\rightarrow \mathbb P^1$. 

We have a fibration $$\Pi_{i,j}:A_{i, j}\rightarrow \bar E_{i,j}$$  over an elliptic curve $\bar E_{i,j}$,  and a  double cover $\sigma_{i,j}:\bar E_{i,j}\rightarrow \mathbb P^1$  branched on four singular values of $\breve\pi_{i,j}$ and   making the following commutative diagram:

\begin{center}
\begin{tikzcd}
A_{i,j}\arrow[r, "\tau_{i,j}"]\arrow[d,"\Pi_{i,j}"']& \breve{\mathcal C}_{i,j}\arrow[d,"\breve\pi_{i,j}"]\\
\bar E_{i,j}\arrow[r, "2:1"', "\sigma_{i,j}"]&\mathbb P^1
\end{tikzcd}
\end{center}

The preimage $E'_{ b_i}=\tau_{i,j}^{-1}(E_{\bold b_i})\subset A_{i,j}$  intersect two distinct points on each fiber of $\Pi_{i,j}$. Every fiber of $\Pi_{i,j}$ is either $\PP^1$ or $\PP^1\setminus\{{\text{two points}}\}$, and there are four fibers which are $\PP^1\setminus\{{\text{two points}}\}$. 
The restriction of $\tau_{i,j}$ to a fiber of $\Pi_{i,j}$ of the form $\mathbb P^1\setminus \{\mbox{two points}\}$  gives a degree 2 morphism to the exceptional curve in a singular fiber of $\pi_{i,j}$.

Therefore $S_{b_i}$ has two components $A_{i, 1}$ and $A_{i, 2}$, both are ruled surface$\setminus\text{\{8 points\}}$ over an elliptic curve. The intersecting curve between $A_{i, 1}$ and $A_{i, 2}$ is the elliptic curve $E'_{b_i}$.

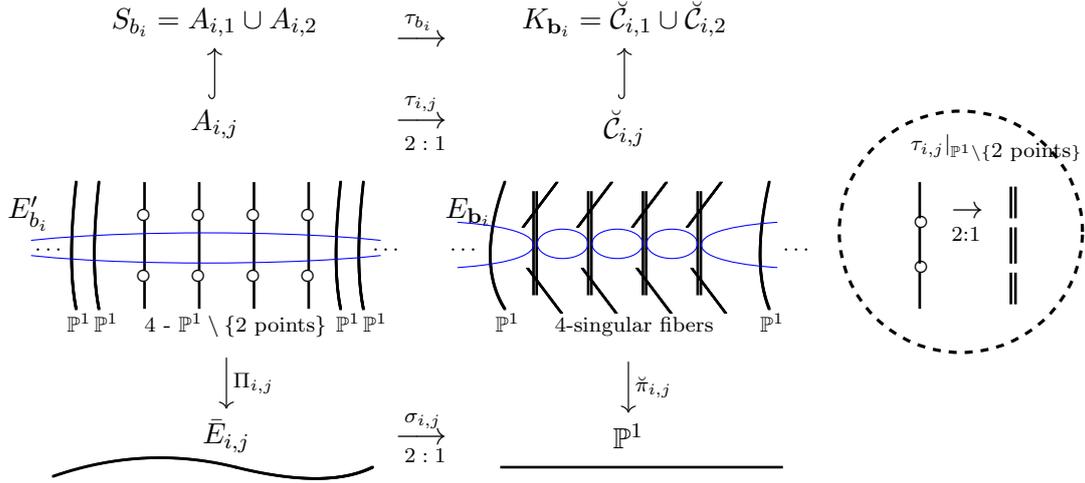
\unitlength 0.6mm 
\linethickness{1pt}
\ifx\plotpoint\undefined\newsavebox{\plotpoint}\fi 
\begin{picture}(100,165)(50,100)

\put(66,182){\tiny$\cdots$}\put(140,182){\tiny$\cdots$}
\put(230,182){\tiny$\cdots$}\put(157,182){\tiny$\cdots$}

\put(90,198){\line(0,-1){6}}\put(88,189){$\circ$}
\put(90,189.5){\line(0,-1){11}}\put(88,175.5){$\circ$}
\put(90,176){\line(0,-1){6}}
	
\put(102,198){\line(0,-1){6}}\put(100,189){$\circ$}
\put(102,189.5){\line(0,-1){11}}\put(100,175.5){$\circ$}
\put(102,176){\line(0,-1){6}}

\put(114,198){\line(0,-1){6}}\put(112,189){$\circ$}
\put(114,189.5){\line(0,-1){11}}\put(112,175.5){$\circ$}
\put(114,176){\line(0,-1){6}}

\put(126,198){\line(0,-1){6}}\put(124,189){$\circ$}
\put(126,189.5){\line(0,-1){11}}\put(124,175.5){$\circ$}
\put(126,176){\line(0,-1){6}}

\put(100,150){\small \begin{tikzcd}\arrow[d,"\Pi_{i,j}"]\\
\bar E_{i,j}
 \end{tikzcd}}

\put(190,150){\small \begin{tikzcd} 
\arrow[d,"\breve\pi_{i,j}"]\\  \mathbb P^1
 \end{tikzcd}}

\qbezier(75,198)(73,180)(75,170)\put(79,165){\tiny $\mathbb P^1$}\put(73,165){\tiny $\mathbb P^1$}\put(167,165){\tiny $\mathbb P^1$}
\qbezier(80,198)(78,180)(80,170)
\qbezier(133,198)(131,180)(133,170)\put(132,165){\tiny $\mathbb P^1$}\put(138,165){\tiny $\mathbb P^1$}\put(225,165){\tiny $\mathbb P^1$}
\qbezier(138,198)(136,180)(138,170)
\put(180,165){\tiny 4-singular fibers}

\multiput(173,188)(.03,.04){250}{\line(0,1){.04}}
\multiput(185,188)(.03,.04){250}{\line(0,1){.04}}
\multiput(197,188)(.03,.04){250}{\line(0,1){.04}}
\multiput(209,188)(.03,.04){250}{\line(0,1){.04}}
\multiput(174,179)(.03,-.04){250}{\line(0,-1){.04}}
\multiput(186,179)(.03,-.04){250}{\line(0,-1){.04}}
\multiput(198,179)(.03,-.04){250}{\line(0,-1){.04}}
\multiput(210,179)(.03,-.04){250}{\line(0,-1){.04}}
\put(176,196){\line(0,-1){23}}\put(175.3,196){\line(0,-1){23}}
\put(188,196){\line(0,-1){23}}\put(187.3,196){\line(0,-1){23}}
\put(200,196){\line(0,-1){23}}\put(199.3,196){\line(0,-1){23}}
\put(212,196){\line(0,-1){23}}\put(211.3,196){\line(0,-1){23}}
\put(90,165){\tiny 4 - $\mathbb P^1\setminus\{\mbox{2 points\}}$}

\qbezier(169,198)(162.625,180)(169,170)
\qbezier(227,198)(223,180)(227,170)

\put(80,222){\small\begin{tikzcd}
S_{b_i}=A_{i,1}\cup A_{i,2}\\
A_{i,j}\arrow[u,hook]
\end{tikzcd}}

\put(170,222){\small\begin{tikzcd}
K_{\bold b_i}=\breve{\mathcal C}_{i,1}\cup \breve{\mathcal C}_{i,2}\\
\breve{\mathcal C}_{i,j}\arrow[u,hook]
\end{tikzcd}}
\put(145,210){$\longrightarrow$}\put(147,215){\tiny $\tau_{i,j}$}
\put(145,228){$\longrightarrow$}\put(147,233){\tiny $\tau_{b_i}$}
\put(147,205){\tiny $2:1$}

\put(168,135){\line(1,0){62}}
\qbezier(70,133)(90,140)(110,135)
\qbezier(110,135)(130,130)(140,135)
\put(145,140){$\longrightarrow$}
\put(147,136){\tiny $2:1$}
\put(147,145){\tiny $\sigma_{i,j}$}

\put(60,190){\small $E'_{b_i}$}

\put(156,190){\small $E_{\bold b_i}$}

\put(260,198){\line(0,-1){8}}\put(258.5,187.5){$\circ$}
\put(260,188){\line(0,-1){8}}\put(258.5,177.5){$\circ$}
\put(260,178){\line(0,-1){8}}

\put(280,197){\line(0,-1){7}}
\put(280,188){\line(0,-1){8}}
\put(280,178){\line(0,-1){7}}

\put(281,197){\line(0,-1){7}}
\put(281,188){\line(0,-1){8}}
\put(281,178){\line(0,-1){7}}

\put(258,205){\tiny $\tau_{i,j}|_{\mathbb P^1\setminus \{\mbox{2 points}\}}$}
\put(267,190){$\rightarrow$}
\put(267,185){\tiny 2:1}

\put(242,160){\begin{tikzpicture}
\begin{scope}[very thick,dashed]
\draw(0,0) circle (1.6cm);\end{scope}
\end{tikzpicture}}

\put(176,181){\begin{tikzpicture}
\node[draw,blue,thin,ellipse,minimum height=5pt,minimum width=19pt](e){};
\end{tikzpicture}} 

\put(188,181){\begin{tikzpicture}
\node[draw,blue,thin,ellipse,minimum height=5pt,minimum width=19pt](e){};
\end{tikzpicture}} 

\put(200,181){\begin{tikzpicture}
\node[draw,blue,thin,ellipse,minimum height=5pt,minimum width=19pt](e){};
\end{tikzpicture}}

\put(212,179){\begin{tikzpicture}
    \draw[blue,thin] plot [smooth,tension=1.5] coordinates{(1,0.3) (0,0) (1,-0.3)};
\end{tikzpicture}}

\put(158.5,179){\begin{tikzpicture}
    \draw[blue,thin] plot [smooth,tension=1.5] coordinates{(0,0.3) (1,0) (0,-0.3)};
\end{tikzpicture}}

\put(65,185){\begin{tikzpicture}
    \draw[blue, thin] plot [smooth,tension=1.2] coordinates{(0,0.1) (2.3,0.2) (4.6,0.1)};
\end{tikzpicture}}

\put(65,180){\begin{tikzpicture}
    \draw[blue, thin] plot [smooth,tension=1.2] coordinates{(0,0.2) (2.3,0.1) (4.6,0.2)};
\end{tikzpicture}}

\end{picture}

Now we are ready to prove Theorem~\ref{level surface}.

\begin{theorem}\label{level} We have the following description of level surfaces of the map $\Phi:T_X^*\to \CC^2_{e_1,e_2}$.
\begin{itemize}
\item[(a)] For every $e\in\CC^2_{e_1,e_2}\setminus\cup_{i=1}^5 \mathbb C\cdot b_i$, $S_e$ is $\bar S_e\setminus\{{\text{16 points\}}}$ where $\bar S_e$ is isomorphic to the Jacobian variety of a curve of genus two. Here, $\mathbb C\cdot b_i$ is the line in $\mathbb C^2_{e_1,e_2}$ through $b_i$ and the origin 0.

\item[(b)] For each $i=1,\ldots,5$, we have the following description of $S_{b_i}$.
\begin{itemize}
    \item[(i)] $S_{b_i}$ consists two irreducible components $A_{i, 1}$ and $A_{i, 2}$. 
    \item[(ii)] Each $A_{i, j}$  is a ruled surface$\setminus\text{\{8 points\}}$ over an elliptic curve $\bar E_{i,j}$.
    \item[(iii)] $A_{i, 1}\cap A_{i, 2}$ is an elliptic curve $E'_{b_i}$.
    \item[(iv)] In the fibration $A_{i, j}\to\bar E_{i,j}$, $E'_{b_i}$ intersects two distinct points at each fiber.
\end{itemize} 
\end{itemize}
\end{theorem}
\begin{proof} By the above argument, we prove (a) for a general $S_e$ and (b). So it is enough to prove that every $S_e$ satisfies (a). If $S_e$ has non-isolated singularities then $K_{\be}$ has also non-isolated singularities. But we know that $K_\be$ has at most isolated singularities if $\be$ does not belong to $\{ \bold b_1, \dots, \bold b_5\}$;
Since the corresponding quadric $Q_\be$ is smooth, we have $K_{\be}$ is irreducible  and $\pi_{\be}:K_\be\rightarrow X$ is a double cover (see Lemma~\ref{l.characterization of irreducible fibers}) which implies that $K_{\be}$ has at worst isolated singularities. 

Therefore $S_e$ has at worst isolated singularities. We also note that $S_e$ is isomorphic to $S_{\lambda e}$ for all $\lambda\in\CC^*$, so $S_e$ is a general singular fiber. Then by using the idea of the characteristic vector fields in \cite{HO09}, $S_e$ should be smooth by the following reason. 


Suppose $q$ is an isolated singularity of $S_e$. Let ${\bold S}=\cup_{\lambda\in\CC^*} S_{\lambda e}$, which is called a vertical surface in \cite{HO09}. Let $z$ be a local coordinate in a neighborhood of $\Phi(S_e)$ in $\Phi({\bold S})=\CC^*$ and consider the Hamiltonian vector fields 
$$\nu_\imath:=\imath_\omega(\Phi^*dz)$$
by identification $\imath_\omega: T_M^*\to T_M$ where $M=T^*_X$ via using the natural symplectic 2-form $\omega$. Since $\Phi$ is a Lagrangian fibration, these vector fields are tangent to $S_e$.  So we have a flow of singularities in $S_e$ coming from the singularity $q$. Therefore $S_e$ cannot have an isolated singularity.

It also implies that $K_\be$ in $\PP(T_X)$ corresponding to $S_e$ is also smooth on $K_{\be}\setminus \cup_{i=1}^{16}\ell_{\be,i}=S_e/\iota$ because $\iota$ acts freely on $S_e$. Furthermore  we can  check that $K_{\be}$ is smooth along each $\ell_{\be,i}$ which implies that $K_{\be}$ is smooth. 
\end{proof}

By the above theorem, we get the following corollaries.

\begin{corollary} The map $\Phi: T_X^* \to\CC^2_{e_1,e_2}$ is flat.
\end{corollary} 
\begin{proof} Clearly, $T_X^*$ is irreducible and $\CC^2_{e_1,e_2}$ is reduced. Then by using our description of level surfaces and Lemma 10.48 in \cite{Kol}, the map $\Phi: T_X^*\setminus \Phi^{-1}(0) \to\CC^2_{e_1,e_2}\setminus\{0\}$ is flat because because $\Phi$ is essentially of finite type, pure dimensional, and its fibers are geometrically reduced. So it is enough to check $\Phi$ is flat over $\{0\}$.

We recall the following diagram of maps.

\unitlength 0.5mm 
\linethickness{1pt}
\ifx\plotpoint\undefined\newsavebox{\plotpoint}\fi 
\begin{picture}(100,100)(30,100)
 \put(90,160)
{$\xymatrix @R=0.5pc @C=1pc {
{\rm Bl}_B\mathbb P(T_X)\ar[r]\ar[rddd]_{\phi}& \mathbb P(T_X)\ar[rd]_{\pi}\ar@{.>}^{\tilde \phi}[ddd]&  &T_X^*\ar@{.>}[ll]\ar[ddd]^{\Phi}\ar[ld]^{\Pi}\\
&  &X & \\
& & &\\
& \mathbb P^1_{e_1,e_2} & &\mathbb C^2_{e_1,e_2}\ar@{.>}[ll]
}$}
\put(143, 177){\tiny $\ell_i'$}
\put(145,170){\tiny $\cap$}

\put(180, 130){\tiny $\ell_i$ }
\put(179, 137){\tiny $\cup$}

\put(203,177){\tiny $S_0$}
\put(204,170){\tiny $\cap$}

\end{picture}

Since $\tilde\phi$ is defined outside $\cup_{i=1}^{16}{\ell_i'}$,   
  $S_0=\Phi^{-1}(0)$ is contained in the union of the zero section of the map $\Phi$ and the preimage of $\cup_{i}^{16}\ell_i'$ of the quotient  map $T_X^*\dashrightarrow \mathbb P(T_X)$.  This implies that  $S_0$ has dimension two. 

For any smooth affine curve $R\subset \mathbb C^2_{e_1,e_2}$ through $0$,  $S_0$ is not an associated point of $\Phi^{-1}(R)$. This implies that the flatness of the map $\Phi$.  
\end{proof}

In the proof of Theorem~\ref{level} it is proved that $K_{\be}$ is smooth for all $\be\in \mathbb P^1_{e_1,e_2}\setminus\{\bold b_1,\ldots,\bold b_5\}$. From this and Lemma~\ref{l.general K} we get the following corollary. 
\begin{corollary}\label{final corollary} For every ${\bold e}\in \mathbb P^1_{e_1,e_2}\setminus\{\bold b_1,\ldots,\bold b_5\}$, 
$K_\be$ is a K3 surface of degree 8 of Kummer type. It has 16 (-2)-curves $\ell_{\be,i}$  which are intersection of  $K_{\be}$ with the exceptional divisor $D$ of $\mu_B:{\rm Bl}_B\mathbb P(T_X)\rightarrow \mathbb P(T_X)$. 
\end{corollary}

\end{document}